\documentclass[10pt]{amsart}
\usepackage{amsfonts}
\usepackage{ mathrsfs }
\usepackage{epsf,graphicx}
\usepackage{amscd}
\usepackage{amsmath,latexsym,amssymb,amsthm}
\usepackage{hyperref}
\usepackage{setspace,cite}
\usepackage{lscape,fancyhdr,fancybox}
\usepackage{a4}
\usepackage{dsfont,texdraw}

\newcounter{cnt1}
\newcounter{cnt2}
\newcounter{cnt3}
\newcommand{\blr}{\begin{list}{$($\roman{cnt1}$)$} {\usecounter{cnt1}
        \setlength{\topsep}{0pt} \setlength{\itemsep}{0pt}}}
\newcommand{\bla}{\begin{list}{$($\alph{cnt2}$)$} {\usecounter{cnt2}
        \setlength{\topsep}{0pt} \setlength{\itemsep}{0pt}}}
\newcommand{\bln}{\begin{list}{$($\arabic{cnt3}$)$} {\usecounter{cnt3}
                \setlength{\topsep}{0pt} \setlength{\itemsep}{0pt}}}
\newcommand{\el}{\end{list}}

\newtheorem{Thm}{Theorem}[section]
\newtheorem{Lem}[Thm]{Lemma}

\newtheorem{Def}[Thm]{Definition}
\newtheorem{Exm}[Thm]{Example}
\newtheorem{Rem}[Thm]{Remark}
\newtheorem{Cor}[Thm]{Corollary}
\newtheorem{Note}[Thm]{Note}

\begin{document}
\title{On Lie algebroid over algebraic spaces}
\author{Ashis Mandal and Abhishek Sarkar}
\begin{abstract}
We consider Lie algebroids over algebraic spaces (in short we call it as $a$-spaces) by considering the sheaf of Lie-Rinehart algebras. We discuss about  properties of universal enveloping algebroid  $\mathscr{U}(\mathcal{O}_X,\mathcal{L})$  of a Lie algebroid $\mathcal{L}$ over an $a$-space $(X, \mathcal{O}_X)$.  This is done by sheafification of the presheaf of universal enveloping algebras for Lie-Rinehart algebras.  We review the extent to which structure of the universal enveloping algebroid of Lie algebroids (over special $a$-spaces) resembles a sheaf of bialgebras. 
In the sequel we present a version of  Poincar\'e-Birkhoff-Witt theorem and Cartier-Milnor-Moore theorem for the Lie algebroid. 
	
\end{abstract}
\maketitle
\section{Introduction}
The notion of Lie algebroid plays a prominent role in geometry as generalized infinitesimal symmetries of spaces 
 related to the corresponding global symmetries of spaces, called Lie groupoid (see \cite{KM}). 
 Lie algebroid over a smooth manifold is a joint generalization of tangent vector bundle over the manifold  and Lie algebras, whose algebraic analogue is known as Lie-Rinehart algebra \cite{UB,MK,LM}. 

There are many instances for consideration  of spaces obtained as the solution space of 
a system of first order partial differential equations (generalized involutive distribution), the vanishing set of a system of holomorphic functions in complex geometry and the vanishing set of a system of polynomials in algebraic geometry. At places, the solution space may not be a smooth object and it is considered as basic analytic space or affine variety. In the presence of singularity, an  analogue of tangent bundle is the tangent sheaf. It is the sheaf of derivations over its structure sheaf. Another important sheaf associated to these spaces is the sheaf of logarithmic vector fields (vector fields tangent to these subspaces). In both cases we find the common underlying structure of coherent sheaf of Lie-Rinehart algebras, which we consider as a Lie algebroid. The Lie algebroid over analytic spaces  and over algebraic varieties are joint generalization of the tangent sheaf \cite{UB,BP,SR} and the sheaf of Lie algebras. We consider Lie algebroids over algebraic spaces $(X, \mathcal{O}_X)$ (in short  $a$-spaces) where $\mathcal{O}_X $ is a sub-sheaf of algebras of the sheaf of continuous functions $C^0$ on $X$, it combines the three types of base spaces as special cases \cite{AM,SR}.

Unlike to the case of (smooth) vector bundles over a smooth manifold, the space of global sections does not capture whole information for a holomorphic (algebraic) vector bundle over a complex manifold (algebraic variety). In the later case, we need to consider sheaf of sections \cite{MK,CW} of the corresponding vector bundles. The necessity is mostly due to lack of the notion of bump functions and partition of unity, which is present automatically for the case of real vector bundles over real smooth manifold. One requires a sheaf-theoretic approach for analytic and algebraic geometry, even for not necessarily singular cases. In doing global calculus over real smooth manifolds, the sheaf-theoretic approach opens up a new way of doing classical differential geometry \cite{SR}. But in the other cases it becomes mandatory except special situations like Stein manifolds and affine varieties. 

The universal enveloping algebra of a Lie-Rinehart algebra is an associative, non-commutative algebra (except the trivial case when it become symmetric algebra).  In \cite{MM}, I. Moerdijk and J. Mr\v{c}un showed that the universal enveloping algebra of a $(\mathbb{K},R)$-Lie-Rinehart algebra has a $R/{\mathbb{K}}$-bialgebra structure. As a special case they deduced $C^\infty(M)/{\mathbb{R}}$-bialgebra structure for Lie algebroid over a smooth manifold $M$.  As in the universal enveloping algebra of a Lie algebra, the bialgebra is equipped with a canonical ascending filtration and  by using the co-algebra structure we get an ascending primitive filtration which coincides for the case of projective Lie-Rinehart algebras. They extended the Cartier-Milnor-Moore theorem from Lie algebras to Lie algebroids over smooth manifolds and more generally for projective Lie-Rinehart algebras.

In this paper we discuss properties of universal enveloping algebroid  denoted by $\mathscr{U}(\mathcal{O}_X,\mathcal{L})$  of a Lie algebroid $\mathcal{L}$ over an $a$-space $ (X, \mathcal{O}_X)$.  Here $\mathscr{U}(\mathcal{O}_X,\mathcal{L})$ is obtained by sheafification of the presheaf of universal enveloping algebras of Lie-Rinehart algebras associated with the space of sections of $\mathcal{L}$. In section $2$, we define the notion of Lie algebroids over algebraic spaces. Then in section $3$, we present a version of Poincare-Birkhoff-Witt theorem for universal enveloping algebroid $\mathscr{U}(\mathcal{O}_X,\mathcal{L})$  of a Lie algebroid $\mathcal{L}$ over $(X, \mathcal{O}_X)$.  In section $4$, we consider the sheaf of logarithmic differential operators and Weyl algebra associated with some principal divisors, namely nilpotent cone and normal crossing divisor. Next, in section $5$, we present a sheaf of generalized bialgebra structure on $\mathscr{U}(\mathcal{O}_X,\mathcal{L})$ which we call $\mathcal{O}_X/{\mathbb{K}_X}$- bialgebra.  In section $6$,  we present a version of Cartier-Milnor-Moore theorem for locally free Lie algebroids over some special type of $a$-spaces. This yields a categorical equivalence between  special subcategories of the Lie algebroids and of the generalized sheaf of bialgebras over special types of $a$-spaces.
\medskip

\section{Lie algebroids over special $a$-spaces}
\subsection{Algebraic spaces or $a$-spaces} 
A class of objects in geometry, for example real $C^k$-manifolds, analytic spaces, algebraic varieties, schemes all are described as locally ringed spaces \cite{CW}. Here, we consider a special type of it which is called algebraic spaces and in short we like to refer as $a$-spaces (see \cite{AM}).  It facilitates to study all the three kinds of geometry (smooth, analytic and algebraic) in a common framework. The notation $\mathbb{K}$ is used for $\mathbb{R, C}$ (the real or complex number fields respectively) or a general algebraically closed field of characteristic zero. We denote by $\mathbb{K}_X$, the constant sheaf on a base space $X$ with stalks being isomorphic to $\mathbb{K}$. The sheaf of $\mathbb{K}$-valued continuous functions on $X$ is denoted by $C^0$.
\begin{Def}{\bf Algebraic space or $a$-space:}
	Let $X$ be a topological space and $\mathcal{O}_X$ be a $\mathbb{K}_X$-subalgebra of the sheaf of $\mathbb{K}$-algebras $C^0$.  In other words, let $\mathcal{O}_X$  be a sheaf of $\mathbb{K}$-algebras where for an open set U $\subset $ X we assign subalgebra $\mathcal{O}_X(U)$ of the algebra $C^0(U)$ and the assignment is compatible with restriction morphisms.
	
	 The pair $(X, \mathcal{O}_X)$ is said to be an algebraic space over the topological space $X$ or, simply $(X, \mathcal{O}_X)$ is an  $a$-space. 
	\end{Def}	 
	 In most of the cases, we consider an $a$-space $(X, \mathcal{O}_X)$ where $\mathcal{O}_X$ is the structure sheaf on $X$.
	 		
A morphism  $(\mathcal{\phi,\phi^*}): (X,\mathcal{O}_X) \Longrightarrow (Y,\mathcal{O}_Y)$ between two $a$- spaces is made up by a continuous map $\phi : X\longrightarrow Y $ and a family of ring homomorphisms 
		\begin{center}
			$\phi_V ^* :\mathcal{O}_Y(V) \longrightarrow \mathcal{O}_X(\phi^{-1}(V))$

		\hspace{.7 cm}	$ g \longmapsto g\circ {\phi}|_{\phi^{-1}(V)} $,
			
				\end{center}
		where $V$ is an open subset of $Y$ and $ g\in \mathcal{O}_Y (V)$. All
		maps commute with inherent restriction maps for the sheaves.
	\begin{Rem}
An appropriate choice for the structure sheaf $\mathcal{O}_X$ provides various examples of $a$-spaces $(X, \mathcal{O}_X)$ coming from smooth manifolds, complex manifolds, analytic spaces and  algebraic varieties. 
\end{Rem}
 \textbf{Tangent Sheaf :}  The tangent sheaf $\mathcal{T}_X$ on a smooth manifold $(X,C^\infty)$ is defined by $\mathcal{T}_X(U) :=Der_{\mathbb{R}}(C^\infty(U))$ for any open subset $U$ of $X$. Elements of  $\mathcal{T}_X(U)$ are linear operators on $C^\infty(U) $ satisfying Leibniz rule, which are equivalent to  smooth  vector fields on $U$. 
 
 Analogous consideration holds for Stein spaces and affine varieties \cite{CW}, in view of the Cartan-Serre theorem \cite{FF}. But, in a more general case for any analytic space or algebraic variety  $(X, \mathcal{O}_X)$, firstly we need to assign $ U \mapsto Der_{\mathbb{K}}(\mathcal{O}_X(U))$ for each Stein open set or affine open set $U$ and it induces a presheaf. By considering $X= \cup_i U_i$ with an open Stein or affine cover, we get a canonical surjective homomorphisms  $Der_{\mathbb{K}}(\mathcal{O}_X(U_i)) \rightarrow Der_{\mathbb{K}}(\mathcal{O}_{X,x})$ of $\mathbb{K}$-vector spaces for each $x \in U_i$.  For  an open set $U (= \cap_i U \cap U_i)$, we define restriction morphism from $U_i$ to $U\cap U_i$. A section $\tilde{D}$ on $U$ is given by a collection of sections $D_i$ associated with good open sets $U \cap U_i$ and using this we can define restriction for the above presheaf. Consequently, the sheaf $ \mathcal{D}er_{\mathbb{K}_X}(\mathcal{O}_X)$ is obtained by sheafification of this presheaf. 
 
 In view of this, the tangent sheaf over an $a$-space $(X, \mathcal{O}_X)$ is given by $ \mathcal{D}er_{\mathbb{K}_X}(\mathcal{O}_X)$- the sheaf of derivations over $ \mathcal{O}_X$. It has a canonical structures of a $\mathbb{K}_X$-Lie algebra and an $\mathcal{O}_X$-module  with the compatibility condition via Leibniz rule.

The \textbf{Cartan-Serre Theorem} states that for any coherent analytic (or algebraic) sheaf $\mathcal{F}$ on a Stein space (or affine variety) $X$ the stalk $\mathcal{F}_x$ of $\mathcal{F}$ at a point $x \in X$ is generated as an $\mathcal{O}_{X,x}$-module by global sections of the sheaf $\mathcal{F}$. Similar results holds for a coherent $\mathcal{O}_X$-module on a smooth manifold $X$.

Next, we consider Lie algebroids over algebraic spaces as a sheaf of Lie-Rinehart algebras equipped with a morphism  to the tangent sheaf on $a$-spaces. Let ${R}$ be a commutative  associative algebra over a field $\mathbb{K}$, then the set $Der_{\mathbb{K}}({R})$ of all $\mathbb{K}$-derivations of ${R}$, is a $\mathbb{K}$-Lie algebra as well as an ${R}$-module. 
            
      A $(\mathbb{K}, R)$-\textbf{Lie-Rinehart algebra} $L$ is a Lie algebra  $L$ (over $\mathbb{K}$) with an $R$-module structure and an $R$-module homomorphism $\rho:L\rightarrow Der_{\mathbb{K}}(R)$ such that 
\begin{itemize}
\item the map $\rho:L\rightarrow Der_{\mathbb{K}}(R)$ is simultaneously an $R$-module homomorphism and a $\mathbb{K}$-Lie algebra homomorphism and 
\item $ [x, ry]=r[x,y]+\rho(x)(r)y~~\mbox{ for all}~ x,y\in L,~r\in R.$
\end{itemize}
\begin{Def} \textbf{Lie algebroids over an $a$-space :}  A Lie algebroid $\mathcal{L}$ over an $a$-space $(X,\mathcal{O}_X)$ is a sheaf of $(\mathbb{K}_X, \mathcal{O}_X)$-Lie-Rinehart algebras.

		That is $\mathcal{L}$ is a sheaf of $\mathbb{K}$-Lie algebras and $\mathcal{O}_X$-modules equipped with a morphism of $\mathcal{O}_X$-modules  $\mathfrak{a}:(\mathcal{L},[\cdot,\cdot])\rightarrow (\mathcal{D}er_{\mathbb{K}_X}(\mathcal{O}_X),[\cdot,\cdot]_c)$, called the anchor map. The map $\mathfrak{a}$ is also a morphism of sheaves of $\mathbb{K}$-Lie algebras  and satisfying the Leibniz rule: 
	 $[D,f D'] = f[D,D']+\mathfrak{a}(D)(f) D'$ for every sections $f \in \mathcal{O}_X$ and $D,D' \in \mathcal{L}$ .
	
	We denote this Lie algebroid as $(\mathcal{L}, [\cdot, \cdot], \mathfrak{a})$ or simply by $\mathcal{L}$.
\end{Def}

\begin{Rem}
Lie algebroids over a smooth manifold $X$ can be viewed  by this general set up of Lie algebroids over the $a$-space $(X, C^\infty)$. 
\end{Rem}
A morphism of Lie algebroids over an $a$-space $$\phi :  (\mathcal{L}_1,[\cdot,\cdot]_1,\mathfrak{a}_1)\rightarrow (\mathcal{L}_2,[\cdot,\cdot]_2,\mathfrak{a}_2)$$ is a sheaf homomorphism of Lie-Rinehart algebras. 
	
	 The sheaves of $(\mathbb{K}_X, \mathcal{O}_X)$-Lie-Rinehart algebras  along with the  sheaf homomorphisms of $(\mathbb{K}_X,\mathcal{O}_X)$-Lie-Rinehart algebras form a category. We call it the category of Lie algebroids over the $a$-space $(X,\mathcal{O}_X)$.
\begin{Thm} \label{Vec bun & loc free}
Let $(X, \mathcal{O}_X)$ be an $a$-space. There is a one-to-one correspondence between vector bundles over $X$ (in the respective  category) and locally free sheaves of $\mathcal{O}_X$-modules of finite rank.
\end{Thm}
In fact, the correspondence mentioned above is an equivalence of categories (see \cite{BP, SR, CW}).
   The classical cases of Lie algebroids can be viewed now as Lie algebroids over $a$-spaces which are locally free sheaves of $\mathcal{O}_X$-modules with finite rank.
Here, we take the $a$-spaces $(X, \mathcal{O}_X)$ as smooth manifold, complex manifold, analytic space, algebraic variety respectively, all are with their associated structure sheaf.

In case of $X$ is a  $C^k$-manifold with $k \in \mathbb{N}$, the tangent sheaf $\mathcal{D}er_{\mathbb{K}_X}(\mathcal{O}_X)$ is trivial. We do not consider this in our discussion. 
 \begin{Note} \label{Schemes} 
For a finitely generated (regular) commutative algebra $R$ over $\mathbb{K}$ we have the affine scheme $X=Spec(R)$ with its structure sheaf of regular functions $\mathcal{O}_X$ induced by localization of the ring $R$. On extending the localization ideas from rings to the modules, we find that $\mathcal{D}er_{\mathbb{K}_X}(\mathcal{O}_X)$ is the induced sheaf by localization of the $R$-module $Der_{\mathbb{K}}(R)$ with standard $(\mathbb{K},R)$-Lie-Rinehart algebra structure. It is a locally finitely presented (free) sheaf of $\mathcal{O}_X$-module.

Lie algebroids over schemes (Noetherian, separated or finite type) have been studied in the literature \cite{UB,CV,AP,MK}. But, approaches for  proofs of necessary theorems can be different for the more general set up. 
\end{Note}
It is well known that by considering the space of global sections, there is special Lie-Rinehart algebra with finitely generated projective module structure being associated with a Lie algebroid in the classical case (for base spaces are smooth manifolds, Stein manifolds and affine varieties) \cite{BRT}. On the other hand, we recall that for a $(\mathbb{K}, R)$-Lie-Rinehart algebra $L$ (finitely generated projective module), we can associate a Lie algebroid (locally free sheaf of finite rank) over the affine scheme $Spec(R)$, which is induced from the localization \cite{CW} of $L$ as an $R$-module, whose space of global sections become $L$. 

\begin{Exm}
For a real smooth manifold (complex manifold) as an $a$-space $(X, \mathcal{O}_X)$, the sheaf of smooth (holomorphic) vector fields is isomorphic to the sheaf of Lie-Rinehart algebras $(\mathcal{D}er_{\mathbb{K}_X}(\mathcal{O}_X), [\cdot,\cdot]_c, id)$. It is a locally free Lie algebroid of rank equals to dimension of the manifold  $X$.

The sheaf of differential $1$-forms $\Omega^1_X$ over a Poisson manifold $X$ has a canonical Lie algebroid structure. For smooth and analytic cases the details can be found in \cite{RF} and \cite{BP} respectively.
\end{Exm}

\begin{Exm} \label{Foliation}
	A singular foliation on a  manifold $(X, \mathcal{O}_X)$ is a sub-sheaf  $\mathcal{F}$ of the sheaf $\mathcal{T}_X$ of $\mathcal{O}_X$-modules, which is $(i)$ stable under the Lie bracket and $(ii)$ locally finitely generated. It provides a generalized involutive distribution on $X$, forms a Lie algebroid over $(X, \mathcal{O}_X)$ (see \cite{CL,RF, BP}).  
\end{Exm}
\begin{Exm}	
Let $\mathbb{K}$ be an algebraically closed field (characteristic zero) and $X= \mathbb{A}^n$ be the affine $n$-space. Consider an (irreduciable) affine algebraic set $Y = V(I)\subset \mathbb{A}^n$ with its co-ordinate ring of regular functions $\mathcal{O}(Y):= \mathbb{K}[x_1, \dots, x_n]/I $, where $I$ is an (prime) ideal in the polynomial algebra $\mathbb{K}[x_1,\dots,x_n]$. We have $\mathfrak{X}(Y) := Der_{\mathbb{K}}(\mathcal{O}(Y))$, the space of derivations over $\mathcal{O}(Y)$ as the analogue of space of algebraic vector fields, which is a $(\mathbb{K}, \mathcal{O}(Y))$-Lie-Rinehart algebra and $\mathfrak{X}^T(X) := \{D \in \mathfrak{X}(X) \mid D(I) \subset I\}$ as the analogue of space of algebraic vector fields that are tangent to $Y$, which is a $(\mathbb{K}, \mathcal{O}(X))$-Lie-Rinehart algebra (both of these are finitely presented module). We can extend this idea by considering it as local model via sheaf theoretic tools to get sheaf of logarithmic derivations over an algebraic variety (locally affine). 
\end{Exm}

\begin{Exm}
	Consider the affine schemes $\mathbb{A}^I: = Spec(\mathbb{K}[x_i]_{i \in I})$ and $\mathbb{A}^\infty := \mathbb{A}^\mathbb{N}$ (see \cite{VD}). The space of global sections of the tangent sheaf for the affine scheme $\mathbb{A}^I$ is the free $\mathbb{K}[x_i]_{i \in I}$-module spanned by the formal derivations $\{ \partial_{x_i}\}_{i \in I}$. This provides an example of Lie algebroids which are quasicoherent $\mathcal{O}_{\mathbb{A}^\mathbb{N}}$-module.
\end{Exm}
\subsection{Analytic spaces} \label{analytic spaces} 
Let $X = \mathbb{C}^n$ be the $n$-dimensional complex manifold and $\mathcal{O}_X$ be the sheaf of holomorphic functions. Let $\mathcal{I}$ be a coherent ideal-sheaf of $\mathcal{O}_X$. The vanishing set $V(\mathcal{I})$ of $\mathcal{I}$ is a subspace of $X$ which is not necessarily a submanifold (it may have singularity as well). Let us denote $Y = V(\mathcal{I})$ and we consider the sheaf of functions $\mathcal{O}_Y := \mathcal{O}_X/{\mathcal{I}}$ as its structure sheaf. Here, the pair $(Y, \mathcal{O}_Y)$ is an analytic space (see in \cite{BP}).
	
If $Y$ is not smooth (singular), it has no tangent bundle in general, then we consider derivations, $\mathcal{T}_Y := \mathcal{D}er_{\mathbb{C}_Y}(\mathcal{O}_Y)$ is the tangent sheaf for $(Y, \mathcal{O}_Y)$. Here $(\mathcal{T}_Y, [\cdot,\cdot]_c,id)$ is a coherent (locally finitely generated) sheaf of Lie-Rinehart algebras over $(Y,\mathcal{O}_Y)$.

We consider the generalized involutive distribution (or singular foliation) given by holomorphic vector fields of $X$ that are tangent to $Y$ \cite{BP} as 
$$\mathcal{T}_X(-log Y) := \{D \in \mathcal{T}_X : D(\mathcal{I}) \subset \mathcal{I}\} \hookrightarrow \mathcal{T}_X$$
with canonical Lie algebroid structure.  For any two sections $D,D' \in \mathcal{T}_X(-log Y)$ and a section $g \in \mathcal{I}$ we have the identity $[D,D']_c(g)=D(D'(g))-D'(D(g)) \in \mathcal{I}$. It is associated with $\mathcal{T}_Y$ via the Lie algebroid epimorphism 
$$\rho : \mathcal{T}_X(-log Y) \rightarrow \mathcal{T}_Y$$
 defined by $\rho (D) = \bar{D}$, $\bar{D}([{f}]) = [{D(f)}]$ for any sections $f \in \mathcal{O}_X, D \in\mathcal{T}_X(-log Y)$.
	
	\bigskip
 \textbf{Nilpotent cone:} \label{Nilpotent cone}
Here we take a family of level sets $\{f^{-1}(c) \mid c \in \mathbb{C}\}$ associated to  a system of first order homogeneous  partial differential equations 	$$D_x(f) = D_y(f) = D_z(f) = 0.$$ For $c = 0$ we get the surface with singularity, which is called a nilpotent cone. It forms an analytic space with a singularity at the origin $\bar{0}=(0,0,0)$. 

One may describe this nilpotent cone through a generalized involutive distribution generated by the Hamiltonian vector fields in a Poisson manifold structure on $X:=\mathbb{C}^3$.
The Poisson algebra $(\mathcal{O}_{\mathbb{C}^3}(\mathbb{C}^3), \{\cdot,\cdot \})$ defined on coordinate functions $x, y, z$ as follows. $$\{x,y\} = 2y, \{x,z\} = -2z, \{y,z\} = x.$$
Then the associated Hamiltonian vector fields  are given by $$D_x = 2y\partial_y - 2z \partial_z ,  D_y = -2y\partial_x + x \partial_z , D_z =  2z\partial_x - x \partial_y.$$
 We consider the system of first order homogeneous  partial differential equations
	$$D_x(f) = D_y(f) = D_z(f) = 0.$$ 
	
	Then we find solution of the above system is $f(x,y,z) = x^2 +4yz$. It describes a family of level sets $f^{-1}(c)$, is  parametrized by $c \in \mathbb{C}$. For each nonzero values of $c$ we get  a $2$-dimensional submanifold.  Now $Y= f^{-1}(0)$ is the vanishing set of an ideal sheaf $\mathcal{I} \subset \mathcal{O}_{X}$ generated by the single element $x^2 + 4yz$. Here, the Lie algebroid $\mathcal{T}_{X}(-log Y)$ is $\mathcal{O}_X$-linear span of the Hamiltonian vector fields tangent to $Y$, and $\mathcal{T}_Y$ is spanned by $D_x|_Y, D_y|_Y, D_z|_Y$ as $\mathcal{O}_Y$-module.  Moreover, the generators of $\mathcal{T}_{X}(-log Y)$ satisfy the equation $$x D_x + 2z D_y + 2y D_z = 0.$$ This shows that it is not a locally free $\mathcal{O}_X$-module. Here, $\mathcal{T}_{Y\setminus \{\bar{0}\}}$ is a locally free module of rank $2$ over $\mathcal{O}_{Y\setminus \{\bar{0}\}}$. However at the origin, this sheaf is no longer free and the stalk $\mathcal{T}_{Y,\bar{0}}$ is a module of rank $3$ over the algebra $\mathcal{O}_{Y,\bar{0}}$.

 It follows that the tangent sheaf $\mathcal{T}_Y$ is not a locally free sheaf of Lie-Rinehart algebras over the analytic space $(Y, \mathcal{O}_Y)$ but it is a coherent sheaf generated by the three elements.
 
\medskip	
\textbf{Normal crossing:} \label{Normal crossing}
	Let $X = \mathbb{C}^2$ with global coordinate functions $x,y$ and we consider the analytic space as $(Y, \mathcal{O}_Y)$ associated to the vanishing set of the ideal sheaf $\mathcal{I}$ generated by the function $xy$. Here the subspace $Y$ is the union of points on the axes, known as a normal crossing divisor. The Lie algebroid  $\mathcal{T}_X(-log Y)$ is $\mathcal{O}_X$-module generated by the derivations $x \partial_x$ and $y \partial_y$.

	  For a point $p$ on $x$-axis or $y$-axis other than the origin $\bar{0}:=(0,0)$, the stalk $\mathcal{T}_{Y,p}$ is generated by the germs between one of the vector fields $x \partial_x$ or $y \partial_y$ at $p$. It follows that the tangent sheaf $\mathcal{T}_Y$ is a locally free $\mathcal{O}_Y$-module of rank $1$ on $Y\setminus \{\bar{0}\}$. But the stalk $\mathcal{T}_{Y,\bar{0}}$ is a module of rank $2$ over $\mathcal{O}_{Y,\bar{0}}$. Therefore, $\mathcal{T}_Y$ is an example for Lie algebroid which is coherent sheaf but not locally free. 
	  Observe that $\mathcal{T}_X(-log Y)$ is a locally free sheaf of $\mathcal{O}_X$ module of rank $2$ on $X\setminus Y$. It is equal to $\mathcal{T}_X$ on $X\setminus Y$ and on $Y$ it is isomorphic to  $\mathcal{T}_X$ but equality not holds.   
	  
	  See Section \ref{not locally free} for details about non-locally free property of the Lie algebroids.
	  


\section{Universal enveloping algebroid of a Lie algebroid }
Here we describe universal enveloping algebroid of a Lie algebroid in terms of the universal enveloping algebra for a Lie-Rinehart algebra. It is obtained through sheafification and by focusing on a model object being the sheaf of differential operators defined on a smooth  manifold. Subsequently, we present a generalized version of Poincare-Birkhoff-Witt theorem (PBW theorem in short hand notation ) for Lie algebroids over some special $a$-spaces, which is analogous to its local version of Lie-Rinehart algebras in \cite{GR}. First we recall some basic concepts related to the sheafifications from \cite{SR}, which help us to follow notations used in later part of our presentation.
 
\textbf{Sheaf associated to a presheaf :} \label{sheafification} Let $\mathcal{F}$ be a presheaf of $\mathcal{O}_X$-modules. At each point $x$, the stalk $\mathcal{F}_x$ has induced algebraic structure from the presheaf.  On sheafification of the presheaf $\mathcal{F}$, we get a sheaf denoted as $\mathcal{F}^\#$ has this structure along with the universal property. Moreover a  section of the $\acute{e}$tale space associated with $S \in \mathcal{F}(U)$ is given as $\tilde{S} : U \rightarrow \displaystyle \cup_{x \in U} \mathcal{F}_x$ defined by $\tilde{S}(x) = [S_x] $ (the equivalence class of germ of $S$ at $x$).
 Now for any section $S' \in \mathcal{F}^\# (U)$ there exist an open cover $\{U_i: i\in I\}$ of $U$ where $S_i \in \mathcal{F}(U_i)$ for $ i \in I$ such that $S'|_{U_i} = \tilde{S_i}$.  Here $S_i = S_j$ on $U_i \cap U_j $ (non empty) and the sheaf $\mathcal{F}^\#$ is locally determined via the presheaf $\mathcal{F}$. Hence, the sheaf $\mathcal{F}^\#$ become an $\mathcal{O}_X$-modules and stalks of $\mathcal{F}$ and $\mathcal{F}^\#$ are same at each point.
 
The assignment $ \mathcal{F} \mapsto \mathcal{F}^\#$ induces a functor from the category of presheaves to the category of sheaves over some topological space. The adjunction of this functor is given by the natural inclusion functor.
Consider the notion $^p(\cdot)$ as $^p(\mathcal{F}^\#)=\mathcal{F}$ for a presheaf $\mathcal{F}$.

 \textbf{Operations on sheaves:} We associate new sheaves using algebraic operations locally followed by  sheafification.
 
 $(i)$ The tensor product $\mathcal{F}_1 \otimes_{\mathcal{O}_X} \mathcal{F}_2$ of any two given $\mathcal{O}_X$-modules $\mathcal{F}_1$ and $\mathcal{F}_2$ is the sheafification of the canonical presheaf $\{U \mapsto \mathcal{F}_1(U) \otimes_{\mathcal{O}_X(U)} \mathcal{F}_2(U)$ for every open set $U \in X$ with componentwise restrictions morphism$\}$. The associated presheaf is denoted as $^p(\mathcal{F}_1 \otimes_{\mathcal{O}_X} \mathcal{F}_2)$.

 $(ii)$ The quotient sheaf $\mathcal{F}_1/{\mathcal{F}_2}$ for a sheaf $\mathcal{F}_1$ with a subsheaf $\mathcal{F}_2$ is the sheafification of the cannonical preasheaf $\{ U \mapsto \mathcal{F}_1(U)/{\mathcal{F}_2(U)}$ for every open set $U$ in $X$ with natural restriction morphism$\}$. The associated presheaf is denoted as  $^p(\mathcal{F}_1/{\mathcal{F}_2})$.
 
 There are two other important sheaves associated with sheaves $\mathcal{F}_1$ and $\mathcal{F}_2$ as $\mathcal{O}_X$-modules. One is  the direct sum $\mathcal{F}_1 \oplus \mathcal{F}_2$ and  the other one is the sheaf homomorphisms $\mathscr{H}om _{\mathcal{O}_X}(\mathcal{F}_1, \mathcal{F}_2)$. These sheaves are obtained directly by the assignments $U \mapsto \mathcal{F}_1(U) \oplus \mathcal{F}_2(U)$ and $U \mapsto Hom _{\mathcal{O}_X(U)}(\mathcal{F}_1(U), \mathcal{F}_2(U))$ respectively. 
 
  
 For a Lie algebroid $\mathcal{L}$ over $(X,\mathcal{O}_X)$, we form the sheaf of tensor algebras $\mathcal{T}_{\mathcal{O}_X}\mathcal{L}$ as the sheafification of the cannonical presheaf: $U \mapsto T_{\mathcal{O}_X(U)} (\mathcal{L}(U))$ of tensor algebras for the $\mathcal{O}_X(U)$-modules $\mathcal{L}(U)$. To get the universal enveloping algebroid  of the Lie algebroid $\mathcal{L}$, we need to consider the quotient sheaf of this sheaf of tensor algebras (see \cite{TS,BP}).  A local version of this quotient sheaf is considered in Section (\ref{Weyl alg}). We follow here another approach to define the universal enveloping algebroid.
\subsection {The sheaf of universal enveloping algebras of  Lie-Rinehart algebras}\label{Uni env alg}
Let $(\mathcal{L},[\cdot,\cdot],\mathfrak{a})$ be  a Lie algebroid  over any of the special $a$-space $(X, \mathcal{O}_X)$.  For each open set $U$ of $X$, we find  the universal enveloping algebra $\mathcal{U}(\mathcal{O}_X(U),\mathcal{L}(U))$ of the $(\mathbb{K},\mathcal{O}_X(U))$-Lie-Rinehart algebra $\mathcal{L}(U)$. It satisfies the compatibility condition that for every open set $V\subset U$ there is a canonical restriction map \vspace{-0.05cm}
$$res_{UV}^\mathcal{U}:\mathcal{U}(\mathcal{O}_X(U),\mathcal{L}(U))\rightarrow \mathcal{U}(\mathcal{O}_X(V),\mathcal{L}(V)) $$ defined by $res_{UV}^\mathcal{U}(f$ $\bar{D}_1 \cdots \bar{D}_n)= f|_V$ $\overline{D_1|_V} \cdots \overline{D_n|_V} $ where $$f\in \mathcal{O}_X(U), D_i\in \mathcal{L}(U),\bar{D}_i = \iota_{\mathcal{L}(U)}(D_i),i=1, \dots ,n$$
and $\iota_{\mathcal{L}(U)}:\mathcal{L}(U) \rightarrow \mathcal{U}(\mathcal{O}_X(U),\mathcal{L}(U))$ is the canonical map.
Here, $\mathcal{U}(\mathcal{O}_X(U),\mathcal{L}(U))$ is an associative unital $\mathbb{K}$-algebra and bimodule over $\mathcal{O}_X(U)$ (\cite{MM}). 

  We find that  $\mathcal{U}(\mathcal{O}_X,\mathcal{L})$ is a  presheaf of associative unital $\mathbb{K}_X$-algebras and also a presheaf of left $\mathcal{O}_X$-modules with the assignment $ U \mapsto \mathcal{U}(\mathcal{O}_X(U),\mathcal{L}(U))$ and restriction morphisms are the maps $res_{UV}^{\mathcal{U}}$ for any two open sets $V, U$ with $V \subset U$.

  We consider its sheafification and the resulting sheaf is denoted by $\mathscr{U}(\mathcal{O}_X,\mathcal{L})$. All the necessary algebraic structures on $\mathscr{U}(\mathcal{O}_X,\mathcal{L})$ obtained via the structures induced on stalks of the presheaf $\mathcal{U}(\mathcal{O}_X,\mathcal{L})$. Thus, we have the canonical embedding $$\iota: \mathcal{U}(\mathcal{O}_X,\mathcal{L}) \hookrightarrow \mathscr{U}(\mathcal{O}_X,\mathcal{L})$$
   of presheaf of $\mathcal{O}_X$-linear $\mathbb{K}$-algebra monorphisms, defined by $\iota(U)(S) = \tilde S$ and $\tilde S(x) = S_x$ (germ of $S$ at $x$) for any section $S \in \mathcal{U}(\mathcal{O}_X(U),\mathcal{L}(U)), x \in U$.

Subsequently, we  form a sheaf monomorphism $\iota_X: \mathcal{O}_X \rightarrow \mathscr{U}(\mathcal{O}_X,\mathcal{L})$ of associative unital $\mathbb{K}_X$ algebras and $\iota_{\mathcal{L}}: (\mathcal{L}, [\cdot,\cdot]) \rightarrow (\mathscr{U}(\mathcal{O}_X,\mathcal{L}), [\cdot,\cdot]_c)$ a $\mathcal{O}_X$-linear sheaf homomorphisms of 
$\mathbb{K}_X$-Lie algebras.
\begin{Rem}\label{Uni prop}
The sheaf of universal enveloping algebras $\mathscr{U}(\mathcal{O}_X, \mathcal{L})$ is charecterized by the following \textbf{universal property}: 

Let $\mathcal{A}$ be a sheaf of unital associative $\mathbb{K}_X$-algebra with sheaf homomorphisms $\phi: \mathcal{O}_X \rightarrow \mathcal{A}$ of $\mathbb{K}_X$-unital algebras and $\psi: (\mathcal{L}, [\cdot,\cdot]) \rightarrow (\mathcal{A}, [\cdot,\cdot]_c)$ of  $\mathcal{O}_X$-linear $\mathbb{K}_X$-Lie algebras such that $\phi(f)\psi(D) = \psi(fD)$ and $[\psi(D), \phi(f)]_c = \phi(\mathfrak{a}(D(f)))$ holds for all sections $f$ of $\mathcal{O}_X$ and $D$ of $\mathcal{L}$. Then there exist a unique presheaf homomorphism of $\mathcal{O}_X$-linear $\mathbb{K}_X$-algebras $^p\widetilde{\psi}: \mathcal{U}(\mathcal{O}_X, \mathcal{L}) \rightarrow \mathcal{A}$.

By the universal property of the sheafification, there exists a unique $\mathcal{O}_X$-linear homomorphism of unital $\mathbb{K}_X$-algebras $\widetilde{\psi} : \mathscr{U}(\mathcal{O}_X, \mathcal{L}) \rightarrow \mathcal{A}$ such that $\widetilde{\psi} \circ \iota_X = \phi$ and $\widetilde{\psi} \circ \iota_{\mathcal{L}} = \psi$.
\end{Rem}
The sheaf  $\mathscr{U}(\mathcal{O}_X, \mathcal{L})$ is called {\bf universal enveloping algebroid of the Lie algebroid $\mathcal{L}$}.
\subsection{Generalized PBW theorem} 
For an open set $U$ of $X$, the algebra $\mathcal{U}(\mathcal{O}_X(U),\mathcal{L}(U))$ has a canonical filtration of $\mathcal{O}_X(U)$-modules \cite{GR,MM}: 
$$\mathcal{O}_X(U) = \mathcal{U}_{(0)}(\mathcal{O}_X(U),\mathcal{L}(U))\subset \mathcal{U}_{(1)}(\mathcal{O}_X(U),\mathcal{L}(U)) \subset \mathcal{U}_{(2)}(\mathcal{O}_X(U),\mathcal{L}(U)) \subset \cdots .$$ 
Here, $\mathcal{U}_{(n)}(\mathcal{O}_X(U),\mathcal{L}(U))$ is spanned by the elements in $\mathcal{O}_X(U)$ and the powers $i_{\mathcal{L}(U)}(\mathcal{L}(U))^m$ for $m = 1,2, \dots ,n$. Thus, we can form sub-presheaves of $\mathcal{O}_X$-module $\mathcal{U}_{(n)}(\mathcal{O}_X,\mathcal{L})$ given by
$\mathcal{U}_{(n)}(\mathcal{O}_X,\mathcal{L})(U) := \mathcal{U}_{(n)}(\mathcal{O}_X(U),\mathcal{L}(U))$. 

 The presheaf $\mathcal{U}(\mathcal{O}_X,\mathcal{L})$ of $\mathbb{K}$ algebras has also a natural filtration of presheaves of $\mathcal{O}_X$-modules: 
$$\mathcal{O}_X = \mathcal{U}_{(0)}(\mathcal{O}_X,\mathcal{L}) \subset \mathcal{U}_{(1)}(\mathcal{O}_X,\mathcal{L}) \subset \mathcal{U}_{(2)}(\mathcal{O}_X,\mathcal{L}) \subset \cdots .$$
On sheafification, these presheaves yield a filtration of $\mathscr{U}(\mathcal{O}_X,\mathcal{L})$. Let $\mathscr{U}_{(n)}(\mathcal{O}_X,\mathcal{L})$ be the associated sheaves of $\mathcal{O}_X$-modules for $n \geq 0$.

The  direct sum of the quotient sheaves associated with each of the consecutive sheaves appear in the above filtration of $\mathscr{U}(\mathcal{O}_X,\mathcal{L})$ forms a sheaf of graded algebras. It is denoted by $gr(\mathscr{U}(\mathcal{O}_X,\mathcal{L}))$. Therefore,
 	$$gr(\mathscr{U}(\mathcal{O}_X,\mathcal{L})) = \bigoplus_{n\geq 0} \mathscr{U}_{(n)}(\mathcal{O}_X,\mathcal{L})/\mathscr{U}_{(n-1)}(\mathcal{O}_X,\mathcal{L}).$$

Note that if a Lie algebroid $\mathcal{L}$  has zero bracket with zero anchor map then $\mathscr{U}(\mathcal{O}_X, \mathcal{L})$ become the sheaf of symmetric algebras $\mathcal{S}_{\mathcal{O}_X} \mathcal{L}$.
If for every $x \in X$, there is an open set $U_x$ in $X$ such that $\mathcal{L}(U_x)$ is a projective left $\mathcal{O}_X(U_x)$-module, then by applying PBW theorem for the $(\mathbb{K}, \mathcal{O}_X(U_x))$-Lie-Rinehart algebra $\mathcal{L}(U_x)$ we get the canonical isomorphisms $\mathcal{O}_X(U_x)$-algebras $$\theta_{U_x} : S_{\mathcal{O}_X(U_x)}(\mathcal{L}(U_x))\rightarrow gr (\mathcal{U}(\mathcal{O}_X(U_x),\mathcal{L}(U_x))).$$ This provides a presheaf of $\mathcal{O}_X$-algebra homomorphism (locally isomorphic)
$$\theta : S_{\mathcal{O}_X}(\mathcal{L})\rightarrow gr (\mathcal{U}(\mathcal{O}_X,\mathcal{L})).$$
\begin{Lem} \label{iso on sheaf level}
	Let $(X, \mathcal{O}_X)$ be an arbitrary $a$-space and $\mathcal{F}_1, \mathcal{F}_2$ be two presheaves of $\mathcal{O}_X$-modules with a presheaf homomorphism $\psi: \mathcal{F}_1\rightarrow \mathcal{F}_2$  inducing stalkwise isomorphism at each point of $X$. Then the induced morphism between corresponding sheaves gives a $\mathcal{O}_X$-module isomorphism
		$\psi^\#:\mathcal{F}^\#_1\rightarrow \mathcal{F}^\#_2$.
	\end{Lem}
\begin{proof} Let $s$ be a section of $\mathcal{F}^\#_1$, that is $s \in \mathcal{F}_1^\#(U)$ for some open set $U$ in $X$. There exists an open cover $\{U_i\}_{i \in I}$ of $U$ with sections $s_i \in \mathcal{F}_1(U_i)$ such that $s|_{U_i}=\tilde{s_i}$ for each $i \in I$ and $s(x)=(s_i)_x$ for each $x\in U_i$. Then $\psi^\#(U)(s)|_{U_i}= \widetilde{\psi_U(s_i)}$ where $\tilde{s_i}=\iota (U_i)(s_i)$  and $\iota : \mathcal{F} \hookrightarrow \mathcal{F}^\#$ given by $s \mapsto \tilde{s}$ defined as $\tilde{s}(x)=[s_x]$ for a section $s$ of the presheaf $\mathcal{F}$. Thus $\psi^\#:\mathcal{F}^\#_1\rightarrow \mathcal{F}^\#_2$ is a morphism. We need to check it is a  bijection.

	One-to-one: Let $U$ be any open subset of $X$ and $s, s'\in \mathcal{F}^\#_1(U)$ such that $\psi^\#(U)(s)- \psi^\#(U)(s')= 0$. Then stalkwise we get $\psi_x (s_x- s'_x)= 0$ for every $x\in U$. The stalkwise isomorphisms provides $s_x= s'_x$ for each $x\in U$. Thus, for each point $x$ there is an open neighborhood $U_x$ of $x$ in $U$ where $s= s'$. Using the local identity property of $\mathcal{F}^\#_1$ for this sections $s, s'$ on this open cover $\{U_x\}_{x\in U}$, we get $s= s'$.\\
	Onto: Let $U$ be any open subset of $X$ and $t\in \mathcal{F}^\#_2(U)$. Then considering germs $t_x$ of $t$ at each point $x\in U$ and by applying the inverses of the isomorphisms $\psi_x$, we get $s_i \in \mathcal{F}_1(U_x)$ for some open neighborhood $U_x$ of $x$ such that $\psi_y((s_i)_y)= t_y$ for all $y\in U_x$. Then by the gluing property of $\mathcal{F}^\#_1$ we can form an element $s\in \mathcal{F}^\#_1(U)$ such that $s|_{U_x}=\tilde{s_i}$, where $\tilde{s_i}(y)= (s_i)_y$ for all $y\in U_x$.
\end{proof}

\begin{Lem}\label{iso on stalk level}
	Let $(X, \mathcal{O}_X)$ be an arbitrary $a$-space and  $\mathcal{F}_1, \mathcal{F}_2$ be two presheaf of $\mathcal{O}_X$-modules with a presheaf homomorphism $\psi: \mathcal{F}_1\rightarrow \mathcal{F}_2$ 	such that $\psi$ is a locally isomorphism. Then $\psi$ induces stalkwise isomorphism.
\end{Lem}

\begin{proof} It is easy to see that on the level of  stalks $\psi$ induces morphisms and we need to check it is a bijection.
	
One-to-one: Let $s_1,s_2 \in (\mathcal{F}_1)_x$ such that $s_1 \neq s_2$ where $x\in X$. Then there exists $S_1,S_2 \in \mathcal{F}_1(V)$ for some open set $V \in X$ such that $(S_1)_x=s_1$ and $(S_2)_x=s_2$. Since $\psi$ is a locally isomorphism, thus there exist an open set $U_x$ around $x$ such that $\psi|_{U_x}: \mathcal{F}_1|_{U_x}\rightarrow \mathcal{F}_2|_{U_x}$ is an isomorphism. Without loss of generality $V$ can be thought as an open subset of $U_x$. Consider $S'_1:=\psi(V)(S_1)$ and $S'_2:=\psi(V)(S_2)$. Our claim is that $(S'_1)_x \neq (S'_2)_x$ that is we need to show that  $S'_1|_{V_x} \neq S'_2|_{V_x}$ for every open set $V_x$ around $x$. Suppose not, then there exist an open set $V_x$ such that $S'_1=S'_2$ on $V_x \subset V$. Consider the one-one property of the isomorphism $\psi(V_x):\mathcal{F}_1(V_x)\rightarrow \mathcal{F}_2(V_x)$. Then it contradicts the fact that $S_1 \neq S_2$ on $V_x$. Therefore, $\psi_x: (\mathcal{F}_1)_x \rightarrow (\mathcal{F}_2)_x$ is one-one for every $x \in X$.\\
Onto: Let $s' \in  (\mathcal{F}_2)_x$. Then there exist $S' \in \mathcal{F}_2(V)$ such that $(S')_x=s'$ for some $V \subset U_x$. By considering onto property  of the isomorphism $$\psi(V):\mathcal{F}_1(V)\rightarrow \mathcal{F}_2(V),$$ we get a pre-image of $S'$ say $S \in \mathcal{F}_1(V)$. We have $\psi_x(S_x)=s'$ by considering the germ $S_x$ for $x \in X$.
\end{proof}

Next, we proceed with $(X, \mathcal{O}_X)$ as one of the special $a$-spaces among smooth manifold, complex manifold, analytic space and algebraic variety.

\begin{Thm} \textbf{Generalized PBW theorem :} If $\mathcal{L}$ is a locally free Lie algebroid  over $(X, \mathcal{O}_X)$ then the graded algebra $gr(\mathscr{U}(\mathcal{O}_{X}, \mathcal{L}))$ is isomorphic to the symmetric algebra 	$\mathcal{S}_{\mathcal{O}_{X}}\mathcal{L}$ $(\mathcal{O}_X$-algebra isomorphism).	
	\end{Thm}
\begin{proof} : If $\mathcal{L}$ is a locally free Lie algebroid over $(X, \mathcal{O}_X)$, then for every point $x\in X$ there exist an open neighborhood  $U_x$ of $x$ in $X$ such that
	\begin{center}
		$\mathcal{L}|_{U_x}\cong  \mathcal{O}_X^{\oplus^I}|_{U_x}$, where $I$ may be infinite.
	\end{center}
	Now, we have the PBW presheaf homomorphism of $\mathcal{O}_X$-algebras 
	\begin{center}
		$\theta:S_{\mathcal{O}_X}\mathcal{L} \rightarrow gr(\mathcal{U}(\mathcal{O}_{X}, \mathcal{L}))$\\
		$f \otimes D_1 \otimes \cdots \otimes D_k \mapsto  \frac{1}{k!} f \displaystyle\sum_{\sigma \in S_k} \bar{D}_{\sigma(1)} \cdots \bar{D}_{\sigma(k)}$
	\end{center}
	where $f$ is a section of $\mathcal{O}_X$ and $D_1, \dots , D_k$ are sections of $\mathcal{L}$.
By using local triviality property of $\mathcal{L}$, on each $U_x$ we have $\mathcal{O}_X|_{U_x}$-algebra isomorphism on restrictions of the presheaves
		$$\theta|_{U_x}:S_{\mathcal{O}_X}\mathcal{L}|_{U_x} \rightarrow gr(\mathcal{U}(\mathcal{O}_{X}, \mathcal{L}))|_{U_x}.$$
It induces a stalkwise isomorphisms of associated presheaves (follows from Lemma \ref{iso on stalk level}). Subsequently using the Lemma \ref{iso on sheaf level}, we get the required isomorphism.
\end{proof}

\begin{Rem} \label{embedding}
 By compairing $\mathcal{O}_X$-modules in degree $1$ of the PBW isomorphism we get $\mathcal{L}\cong \mathscr{U}_1(\mathcal{O}_X, \mathcal{L})/{\mathcal{O}_X}$. In turn, we have the short exact sequence (s.e.s) of $\mathcal{O}_X$-modules 
\begin{center}
	$0 \rightarrow \mathcal{O}_X \hookrightarrow  \mathscr{U}_{(1)}(\mathcal{O}_X, \mathcal{L}) \rightarrow \mathcal{L} \rightarrow 0$.
\end{center}
Infact it is also an s.e.s of Lie algebroids with respect to canonical Lie algebroid structures. 
It splits as $\mathcal{O}_X$-modules that is $\mathscr{U}_{(1)}(\mathcal{O}_X, \mathcal{L}) \cong \mathcal{O}_X \oplus \mathcal{L}$, as the Lie algebroid $\mathcal{L}$ is locally free $\mathcal{O}_X$-module. This induces an embedding $$\mathcal{L} \hookrightarrow \mathscr{U}_{(1)}(\mathcal{O}_X, \mathcal{L}) \subset \mathscr{U}(\mathcal{O}_X, \mathcal{L}).$$
\end{Rem}

\textbf{Notation :} From now on we consider the special type of open cover $\{U_x \mid x\in X \}$ of $X$ for a locally free Lie algebroid $\mathcal{L}$ and here the restrictions $\mathcal{L}|_{U_x}$ are free $\mathcal{O}_X|_{U_x}$-modules for every $U_x$.
\begin{Exm}
	 For $\mathcal{L}=\mathcal{T}_X$ the tangent sheaf on a smooth manifold or complex manifold $(X, \mathcal{O}_X)$, the universal enveloping algebroid  $\mathscr{U}(\mathcal{O}_X, \mathcal{T}_X)$ is the usual sheaf of regular differential operators \cite{MK, BP, SR}. 
	
	\label{Uni env alg exam for LR} Let $X$ be an $n$-dimensional $C^\infty$-manifold or a Stein manifold. Then $L:= Der_{\mathbb{K}}(\mathcal{O}_X(X))$ is the $(\mathbb{K}, \mathcal{O}_X(X))$-Lie-Rinehart algebra of first order homogeneous differential operators on $X$ and its universal enveloping algebra $\mathcal{U}(\mathcal{O}_X(X), L)$ is isomorphic to the $\mathbb{K}$-algebra (and $\mathcal{O}_X(X)$-module) of differential operators on $X$. The associated sheaf is denoted as $\mathcal{D}iff(\mathcal{O}_X)$ or $\mathcal{D}_X$ \cite{MK,SR}. In a local coordinate system $(U, (x_1, \dots ,x_n))$, elements of $\mathcal{U}_{(k)}(\mathcal{O}_X(U), Der_{\mathbb{K}}(\mathcal{O}_X(U)))$ are of the form
	$$\displaystyle \sum_{i_1 + \dots +i_n \leq k}f_{i_1 \dots i_n}(\partial/\partial x_1)^{i_1} \circ \cdots \circ (\partial/\partial x_n)^{i_n}.$$
		This approach locally holds for a complex manifold as it has an open Stein cover \cite{BRT}.

For a Lie algebroid $\mathcal{L}$ over a manifold $(X, \mathcal{O}_X)$, the anchor map $\mathfrak{a}: \mathcal{L} \rightarrow \mathcal{T}_X$ induce a cannonical map  $\tilde{\mathfrak{a}}: \mathscr{U}(\mathcal{O}_X, \mathcal{L}) \rightarrow \mathcal{D}_X$.

\end{Exm}

\section{Logarithmic differential operators for principal divisors }
In this section, we recall logarithmic derivations associated with principal divisors in \cite{RK,BP,LM} and discuss  associated universal enveloping algebroids.

We use the notation $ \langle \{D_1, \dots , D_n\} \rangle$ for $R$-module generated by the elements $D_1, \dots , D_n  \in E$, where $R$ is a $\mathbb{K}$-algebra and $E$ is an $R$-module.
\medskip

{\bf Principal divisors:}
Let $(X, \mathcal{O}_X)$ be a complex manifold. A principal divisor consists of a line bundle $N \overset{\pi}{\rightarrow} X$ equipped with a section $s$ whose zero set $Y:=s^{-1}(0)$ is nowhere dense in $X$. It corresponds to a principal ideal sheaf $\mathcal{I}$ which is locally free. If $s|_{U_{\alpha}}= f_{\alpha}$ on some Stein open cover $\{U_{\alpha}\}_{{\alpha} \in I}$ of $X$ then $\mathcal{I}|_{U_{\alpha}}=\langle f_{\alpha} \rangle$ for $ {\alpha} \in I$. In this case, $\mathcal{I}$ is a locally free $\mathcal{O}_X$-module of rank $1$ and so $Y=V(\mathcal{I})$ is a codimension $1$ submanifold of $X$. 
We denote by $N_Y$ the normal bundle of $Y$ in $X$.  Then $TX|_Y \cong TY \oplus N_Y$.
The sheaf of sections $\mathcal{N}_Y$ of the normal bundle of $Y$ is locally given by $ \langle \{ \nabla (f_{\alpha}|_Y) \} \rangle$ on $V_{\alpha}= Y\cap U_{\alpha}$ with $\alpha \in I$.

More generally, the subspace $Y:=V(\mathcal{I}) \subset X$ associated with a principal ideal sheaf $\mathcal{I} \subset \mathcal{O}_X$, considered as a principal divisor. Similar consideration hold in complex algebraic geometry context.

\textbf{Logarithmic derivations :} 
The sheaf $\mathcal{T}_X(- log Y )$ of derivations (vector fields) tangent to an analytic subspace (algebraic variety) $Y= V(\mathcal{I})$ or, the sheaf of derivations of $X$ which preserves the ideal sheaf $\mathcal{I}$ is called the sheaf of logarithmic derivations (Section \ref{analytic spaces}). It has a canonical Lie algebroid structure over $(X, \mathcal{O}_X)$.

\begin{Exm}
Consider the ideal sheaf $\mathcal{I}= \langle xy \rangle$ and $Y= V(\mathcal{I})$. Here 
$(Y, \mathcal{O}_Y)$ is an analytic space (affine algebraic set) with a singularity at the origin. It follows that  $\mathcal{T}_X( - log Y )= \langle \{x\partial_x, y\partial_y\} \rangle $ and the corresponding sheaf of meromorphic $1$-forms with poles along the divisor $Y$ is $\Omega_X^1( log Y)= \langle \{dx/x, dy/y\} \rangle.$
\end{Exm}
\subsection{Algebra of logarithmic differential operators :} \label{Weyl alg}
Next we describe universal enveloping algebroid associated with a Lie algebroid of logarithmic derivations. 

Algebra of differential operators $Diff_{\mathbb{K}}(R)$ for a commutative $\mathbb{K}$-algebra $R$ appeared in \cite{TS} is also denoted  as  $\mathcal{D}_{R/{\mathbb{K}}}$ (e.g. \cite{LM}). Here we compute its counter part associated to special divisors,  previously described (Section \ref{analytic spaces}). For this we use following identity holds in $Diff_{\mathbb{K}}(R)$:
$[D,f]_c= D \circ f - f \circ D= D(f)$ for any $D \in Der_{\mathbb{K}}(R)$ and $f \in R$. 

Recall that under additional algebraic condition on $R$, the canonical $\mathbb{K}$-algebra epimorphism  $$Diff_{\mathbb{K}}(R) \rightarrow \mathcal{U}(R, Der_{\mathbb{K}}(R))$$
  induces an isomorphism. For instance,  the algebra $R$ is $C^\infty(M)$ or $\mathbb{K}[x_1, \dots , x_n]$ or any arbitrary regular $\mathbb{K}$-algebra, where $Der_{\mathbb{K}}(R)$ is a projective $R$-module (see \cite{HM} and \cite{TS}).
\medskip

 {\bf Normal crossing divisors :} 
 Here, we recall the normal crossing (Section \ref{analytic spaces}) as an affine algebraic set. We consider $X=\mathbb{C}^2$ as an algebraic variety, $\tilde{R}:=\mathcal{O}_X(X)= \mathbb{C}[x,y]$. Suppose  the ideal $I=\langle xy \rangle \subset\mathcal{O}_X(X)$. We take the affine algebraic set $Y:=V(I)$ with its algebra of regular functions $R:=\mathcal{O}_Y(Y)=\mathbb{C}[x,y]/{\langle xy \rangle}$. 

The $(\mathbb{C},R)$-Lie-Rinehart algebra $L:=\mathcal{T}_X(-log Y)(X)=\langle \{x\partial_x, y\partial_y\} \rangle $. It is the space of logarithmic derivations  associated to the divisor $Y$, denoted sometime by $Der_{\mathbb{C}}(-log I))$. 
Let $\mathbb{C} \langle x_1,x_2,D_1,D_2 \rangle$ denote the $\mathbb{C}$-algebra of noncommutative polynomials on variables $x_1,x_2,D_1,D_2$.
Here, we have a canonical isomorphism
  $$\mathcal{U}(\tilde{R}, L)\cong \frac{\mathbb{C}\langle x_1,x_2,D_1,D_2 \rangle}{\langle \{[x_i,x_j]_c, [D_i,x_j]_c- \delta_{ij}x_j,[D_i, D_j]_c; i,j\in\{1,2\}\} \rangle}. $$
The normal crossing divisor is a free divisor. Thus, $\mathcal{D}_{\tilde{R}}(-log I)$-the
algebra of  logarithmic differential operators on $X$ associated to the divisor $Y$  is isomorphic to $\mathcal{U}(\tilde{R},L)$ (\cite{LM}).
 By localizations of the $\tilde{R}$-module $L$, we can form the $\mathcal{O}_X$-module of logarithmic derivations $\mathcal{L}$ 
and  deduce the sheaf of logarithmic differential operators $\mathscr{U}(\mathcal{O}_X,\mathcal{L})$ for the free divisor $Y$ corresponding to the algebra $\mathcal{U}(\tilde{R},L)$.

This approach helps us to determine the Weyl algebra $ \mathscr{U}(\mathcal{O}_Y,\mathcal{T}_Y)$ for an analytic space or an algebraic variety $(Y,\mathcal{O}_Y)$. But for affine case, i.e., for basic analytic space or affine variety $Y$ we need only to consider the space of global sections as Weyl algebra, which is $\mathcal{U}(R,\mathcal{T}_Y(Y))$.
\medskip

{\bf Nilpotent cone : } We recall the example from 
Section \ref{analytic spaces}, the nilpotent cone $Y=V(I)$ in $X = \mathbb{C}^3$ where $I= \langle x^2+4yz \rangle \subset \mathcal{O}_X(X)=\tilde{R}$ is a principal ideal. The  Hamiltonian vector fields $D_x, D_y, D_z$ describe a coherent involutive subsheaf $\mathcal{L}:=\mathcal{T}_X(-log Y)$ of the tangent sheaf $\mathcal{T}_X$
 but not locally free. 

 The space of global sections  $L:=\mathcal{L}(X)=Der_{\mathbb{C}}(-log I)$ is a $(\mathbb{C}, \tilde{R})$-Lie-Rinehart algebra.  Here $L$ is a finitely presented $\tilde{R}$-module generated by $\{D_x,D_y,D_z\}$. 
 
 In view of $Y$ as an affine variety inside the affine space $X$, we get $$\tilde{R}=\mathbb{C}[x,y,z], ~~R=\mathbb{C}[x,y,z]/ \langle x^2+4yz \rangle.$$  Consequently, we have a canonical isomorphism (denote it by $\phi$)
 $$\mathcal{U}(\tilde{R}, L)\cong \frac{\mathbb{C}\langle x_1,x_2,x_3,D_1,D_2,D_3 \rangle}{\langle 	\{[x_i,x_j]_c, [D_i,x_j]_c-\{x_i,x_j\},[D_i,D_j]_c-D_{\{x_i,x_j\}}; i,j\in\{1,2,3\}\} \rangle}, $$
 where $\{x_i,x_j\}:= \phi^{-1}(\{\phi(x_i), \phi(x_j)\})$ and $D_{\{x_i,x_j\}}:=\phi^{-1}(D_{\{\phi(x_i),\phi(x_j)\}})$,~for $i,j\in\{1,2,3\}$.

The $(\mathbb{C},\tilde{R})$-Lie-Rinehart algebra $ Der_{\mathbb{C}}(-log I)$ induces a $(\mathbb{C},R)$-Lie-Rinehart algebra structure on the $R$-module generated by the set $\{D_x|_Y, D_y|_Y, D_z|_Y\}$. It is equal to $\mathcal{T}_Y(Y)$. Then the Weyl algebra for $Y$ is 
$$\mathcal{U}(R,Der_{\mathbb{C}}(R)) \cong \mathcal{U}(\tilde{R}, Der_{\mathbb{C}}(-log I))/{Ker(\psi)}$$ where the map $$\psi:\mathcal{U}(\tilde{R}, Der_{\mathbb{C}}(-log I)) \rightarrow \mathcal{U}(R,Der_{\mathbb{C}}(R))$$
 is defined on the generators by the canonical restrictions. Also, we get an analogous  result associated with logarithmic differential operators:
 $$\mathcal{D}_{R/{\mathbb{C}}}\cong\frac{\mathcal{D}_{\tilde{R}}(-log I)}{Ker(\psi)}.$$ 
\subsection{PBW isomorphism for some non-locally free Lie algebroids :}  
In this part, we show that generalized PBW isomorphism holds  for Lie algebroid over $(X,\mathcal{O}_X)$ without the condition of locally free \cite{TS,BP}. 
Here, we consider Lie algebroids for special divisors.
\begin{itemize}
\item For a $(\mathbb{K},R)$-Lie-Rinehart algebra $L$, we have the PBW map (also called the symmetrization map)
\begin{center}
	$\phi: S_RL\rightarrow \mathcal{U}(R,L)$ given by $D_1\otimes \cdots \otimes D_n \mapsto \frac{1}{k!}\displaystyle\sum_{\sigma\in S_k} \bar{D}_{\sigma(1)}\cdots \bar{D}_{\sigma(k)}$
\end{center}
where $D_1, \dots ,D_k \in L$, is a surjective $R$-linear $\mathbb{K}$-algebra homomorphism. The morphism $\phi$ induces a surjective $R$-algebra homomorphism $$\tilde{\phi}:S_RL\rightarrow gr(\mathcal{U}(R,L)).$$

\item The map $\tilde{\phi}$ is an isomorphism when $L$ is, in addition, a projective $R$-module. 

Both of these results are given by G. Rinehart in\cite{GR}.

\item For a locally free Lie algebroid $\mathcal{L}$ over $(X,\mathcal{O}_X)$, at each point $x\in X$ the stalk $\mathcal{L}_x$ is free $\mathcal{O}_{X,x}$-module of fixed finite rank. It provides the $\mathcal{O}_X$-algebra isomorphism $\theta:\mathcal{S}_{\mathcal{O}_X}\mathcal{L}\rightarrow gr(\mathscr{U}(\mathcal{O}_X,\mathcal{L}))$.

\item Suppose $\mathcal{L}$ is a non-locally free Lie algebroid over $(X,\mathcal{O}_X)$, whose stalks $\mathcal{L}_x$ are free $\mathcal{O}_{X,x}$-modules of finite ranks and $x \mapsto rank(\mathcal{L}_x)$ is not a constant function. Hence, $\mathcal{L}_x$ is a $(\mathbb{K},\mathcal{O}_{X,x})$-Lie-Rinehart algebra and a free $\mathcal{O}_{X,x}$-module for each $x\in X$. It satisfies Rinehart's PBW criteria on the level of stalks. Therefore, the induced map of $\theta$ on each stalk provides an isomorphism of $\mathcal{O}_{X,x}$-algebras $\theta_x:\mathcal{S}_{\mathcal{O}_{X,x}}\mathcal{L}_x\rightarrow gr(\mathcal{U}(\mathcal{O}_{X,x},\mathcal{L}_x))$. As the map $\theta$ is defined via sheafification of the morphism induces from $\theta_x$, the PBW map is an isomorphism of $\mathcal{O}_X$-algebras.
\end{itemize}
 Let $\mathcal{L}_x$ be a $(\mathbb{K},\mathcal{O}_{X,x})$-Lie-Rinehart algebra which is also free $\mathcal{O}_{X,x}$-module for every $x\in X$. In view of  Rinehart's PBW theorem we have $\mathcal{O}_{X,x}$-algebra isomorphism by the symmetrization map $$\psi_x:\mathcal{S}_{\mathcal{O}_{X,x}}\mathcal{L}_x\rightarrow gr(\mathcal{U}(\mathcal{O}_{X,x},\mathcal{L}_x)) $$ for every $x\in X$. 
 
 The generalized PBW morphism $\theta:\mathcal{S}_{\mathcal{O}_X}\mathcal{L}\rightarrow gr(\mathscr{U}(\mathcal{O}_X,\mathcal{L}))$ is surjective and induces stalkwise isomorphisms of $\mathcal{O}_{X,x}$-algebras 
$$\theta_x:\mathcal{S}_{\mathcal{O}_{X,x}}\mathcal{L}_x\rightarrow gr(\mathcal{U}(\mathcal{O}_{X,x},\mathcal{L}_x))$$  for each $x\in X$.
 We have $\psi_x=\theta_x$  for all $x\in X$.
 \medskip
 
\subsubsection{On PBW isomorphism for some  examples } \label{not locally free}
 We show that PBW isomorphism holds for the Lie algebroid 
 of logarithmic derivations $\mathcal{L}:=$$\mathcal{T}_X(-logY)$ and tangent sheaf $\mathcal{T}_Y$ associated to $Y$ the nilpotent cone and the normal crossing divisor respectively (defined in Section \ref{analytic spaces}).  

{\bf The PBW isomorphism for nilpotent cone $Y$:}
\begin{itemize}
\item The equation $xD_x+2zD_y+2yD_z=0$ is satisfied by the logarithmic derivations $D_x,D_y,D_z$ associated to $Y$. Therefore, $\mathcal{L}(U_p)=\mathcal{T}_X(-logY)(U_p)$ is a free $\mathcal{O}_X(U_p)$-module of rank $2$ on open neighborhood $U_p$ of $p$ in $X$ with $p\in X \setminus \{\bar{0}\}$. For $p=\bar{0}$ we get that $(\mathcal{L})_{\bar{0}}$ is a free $\mathcal{O}_{X,\bar{0}}$-module of rank $3$. Consequently, $\mathcal{L}$ is a Lie algebroid with stalks $\mathcal{L}_x$ is of rank $2$ or $3$ as a free $\mathcal{O}_{X,x}$-module when $x\neq \bar{0}$ or $x=\bar{0}$ respectively. In this case, we get PBW isomorphism for the Lie algebroid $\mathcal{L}$.

\item Next, we consider the tangent sheaf $\mathcal{T}_Y$ of the analytic space  $(Y,\mathcal{O}_Y)$,
i.e. $\mathcal{D}er_{\mathbb{C}_Y}(\mathcal{O}_Y)=\langle \{[D_x|_Y], [D_y|_Y], [D_z|_Y]\} \rangle $. For $p\in Y \setminus \{\bar{0}\}$ and an open set $U_p$ containing $p$, we consider an open set $V_p:=Y\cap U_p$ in $Y$ containing $p$. Here the restriction, $\mathcal{T}_Y|_{V_p}$ is a free $\mathcal{O}_Y|_{V_p}$-module of rank $2$. Hence, $\mathcal{T}_{Y,p}$ is a free $\mathcal{O}_{Y,p}$-module of rank $2$.
But,  for $p=\bar{0}$ with an open neighborhood $V$ of $p$ in $Y$, the space of derivations  $\mathcal{T}_Y(V)$ cannot be generated by two generators. In fact, $\mathcal{T}_{Y,\bar{0}}$ is a free $\mathcal{O}_{Y,\bar{0}}$-module of rank $3$. We get the PBW isomorphism for the  Lie algebroid of the tangent sheaf $\mathcal{T}_Y$.
 \end{itemize}
 {\bf The PBW isomorphism for normal crossing divisor $Y$:}
 \begin{itemize} 
\item
Let $p\in X \setminus Y$ and $U_p$ be an open set containing $p$ in $X \setminus Y$. Now $\mathcal{T}_X(-logY)|_{U_p}$ is given by the restriction $\mathcal{T}_X|_{U_p}$ as $\mathcal{O}_X|_{U_p}$-modules.
For $q\in \{x-axis \}\setminus \{\bar{0}\}$, we have $U_q:=X\setminus \{x= 0\}$ is an open subset containing $q$. Here  $\partial_y \notin \mathcal{T}_X(-log Y)(U_q)$ on $U_q$ as $\frac{1}{y}$ is not defined on $U_q$. 
We find that $\mathcal{T}_X(-log Y)|_{U_q}$ is a free $\mathcal{O}_X|_{U_q}$-module of rank $2$. 
 Similarly, for a point $q \in \{y-axis \} \setminus \{\bar{0}\}$, there exists an open set $U_q$ containing $q$ with the analogous property. Also, stalk at $\bar{0}$ for $\mathcal{T}_X(-log Y)$ is of rank $2$ and spanned by the stalks of the elements in the generating set. Thus,  $\mathcal{T}_X(-log Y)$ is a locally free $\mathcal{O}_X$-module of rank $2$. 
 It is important to note here that for the singular space $Y$, associated Lie algebroid $\mathcal{L}$ is locally free, is called a free divisor. The generalized PBW map for $\mathcal{L}$ provides an isomorphism.

\item Next, we consider the Lie algebroid $\mathcal{T}_Y$.  For an open set containing $q$ say $V_q:= Y\cap U_q$, we get $\mathcal{T}_Y|_{V_q}=\langle \{(x\partial_x)|_Y\} \rangle$. Thus, $\mathcal{T}_Y|_{V_q}$ is a free $\mathcal{O}_Y|_{V_q}$-module of rank $1$. Similarly $\mathcal{T}_Y|_{V_p}$ is a free $\mathcal{O}_Y|_{V_p}$-module of rank $1$. But, the stalk $\mathcal{T}_{Y,\bar{0}}=\langle \{(x\partial_x)|_Y,(y\partial_y)|_Y\} \rangle$ is a $\mathcal{O}_{Y,\bar{0}}$-module of rank $2$. Therefore we cannot apply generalized PBW theorem for $\mathcal{T}_Y$, since it is not locally free $\mathcal{O}_{Y}$-module but its stalks at any point $q$ is free $\mathcal{O}_{Y,q}$-module with rank either $1$ or $2$. In this case also we get the PBW isomorphism.

\end{itemize}
\section{$\mathcal{O}_X/{\mathbb{K}_X}$-bialgebras}
The notation $\mathcal{O}_X/{\mathbb{K}_X}$-bialgebra is considered in this section as a sheaf of generalized bialgebra structure starting with a sheaf of unital associative $\mathbb{K}_X$-algebra which extends the sheaf $\mathcal{O}_X$. 
As an example, we consider universal enveloping algebroid $\mathscr{U}(\mathcal{O}_X,\mathcal{L})$ for a Lie algebroid $\mathcal{L}$ on $(X, \mathcal{O}_X)$.
On an open set we get back  the $R/{\mathbb{K}}$-bialgebra structure defined on the universal enveloping algebra. For this $R/{\mathbb{K}}$-bialgebra one can recover the Lie-Rinehart algebra by considering the space of primitive elements  (see \cite{MM}). 
In complex and algebraic geometry set up, this  bialgebra structure appeared in \cite{AP, CV, MK}, but we do not find construction involving the space of primitive elements through which we can recover the Lie algebroid. 

\subsection{ $\mathcal{O}_X/{\mathbb{K}_X}$-bialgebras :} Let $(X,\mathcal{O}_X)$ be an $a$-space. Suppose $\mathcal{A}$ is a sheaf of unital associative $\mathbb{K}_X$-algebra extending the sheaf $\mathcal{O}_X$ i.e., $\mathcal{O}_X$ is a subsheaf of algebras of $\mathcal{A}$. We may view $\mathcal{A}$ as a sheaf of $\mathcal{O}_X-\mathcal{O}_X$-bimodules. Let us recall the notation that  $^p(\mathcal{A}\otimes_{\mathcal{O}_X} \mathcal{A})$ is a presheaf of left $\mathcal{O}_X$-modules. 
 We consider the presheaf homomorphism 
$$\psi :~ ^p(\mathcal{A}\otimes_{\mathcal{O}_X} \mathcal{A}) \rightarrow \mathcal{H}om_{\mathbb{K}_X}(\mathcal{O}_X,~ ^p(\mathcal{A}\otimes_{\mathcal{O}_X} \mathcal{A})) $$ defined by 
$\psi(U)(a \otimes b)(f) = a \otimes (bf - fb)$ where $f \in \mathcal{O}_X(U)$ and  $a,b \in \mathcal{A}(U)$, for an open subset  $U$ of $X$. 

The sheafification of this presheaf morphism $\psi$ yields the sheaf of $\mathbb{K}_X$-algebras expressed as $$\mathcal{A}\bar{\otimes}_{\mathcal{O}_X} \mathcal{A}:= \mathscr{K}er  (\psi^\#) = ( \mathcal{K}er ~  \psi)^\# =  (^p(\mathcal{A}\bar{\otimes}_{\mathcal{O}_X} \mathcal{A}))^\# .$$


\begin{Def} 
A tuple $(\mathcal{A}, \Delta,\epsilon)$  is called $\mathcal{O}_X/{\mathbb{K}_X}$-bialgebra if the following conditions hold.

\begin{itemize}
\item $\mathcal{A}$ is a sheaf of unital associative $\mathbb{K}_X$-algebra extending the sheaf $\mathcal{O}_X$;
\item  $\mathcal{A}$ is equipped with the morphism of sheaf of $\mathcal{O}_X$-modules the comultiplication 
 $\Delta : \mathcal{A}\rightarrow \mathcal{A}\otimes_{\mathcal{O}_X} \mathcal{A}$ and counit $\epsilon : \mathcal{A}\rightarrow \mathcal{O}_X$;

\item $\Delta(\mathcal{A})$ is a subsheaf of the sheaf of $\mathcal{O}_X$-modules $\mathcal{A}\bar{\otimes}_{\mathcal{O}_X} \mathcal{A}$.
\item  $\Delta(1) = 1\otimes 1$ and $\epsilon(1) = 1$, where $1$ is the unit of $ \mathcal{O}_X$;
\item $\Delta(ab) = \Delta(a) \Delta(b)$,
 $\epsilon(ab) = \epsilon(a\epsilon(b))$ for any two sections $a,b$ in
 $\mathcal{A}$. 

\end{itemize}
\end{Def}
\begin{Def} 
Let $(\mathcal{A}_1, \Delta_1, \epsilon_1)$ and $(\mathcal{A}_2, \Delta_2, \epsilon_2)$ be two $\mathcal{O}_X/{\mathbb{K}_X}$-bialgebras. A morphism $$\psi : (\mathcal{A}_1, \Delta_1, \epsilon_1)\rightarrow (\mathcal{A}_2, \Delta_2, \epsilon_2) $$ between the sheaves of $\mathcal{O}_X/{\mathbb{K}_X}$-bialgebras is a \\
$(a)$ morphism of sheaves of unital $\mathbb{K}_X$-algebras, and \\
$(b)$ morphism of sheaves of coalgebras in the category of left $\mathcal{O}_X$-modules.
\end{Def} 
It follows that the $\mathcal{O}_X/{\mathbb{K}_X}$-bialgebras and their morphisms forms a category, we call it the category of $\mathcal{O}_X/{\mathbb{K}_X}$-bialgebras. 
\newpage
\begin{Exm}
	\textbf{The universal enveloping algebroid $\mathscr{U}(\mathcal{O}_X,\mathcal{L})$ as $\mathcal{O}_X/{\mathbb{K}_X}$-bialgebra:} \label{universal alg} 
	
Let $(\mathcal{L},[\cdot,\cdot],\mathfrak{a})$ be a Lie algebroid over one of the special $a$-space $(X,\mathcal{O}_X)$. Consider the presheaf $U \mapsto (\mathcal{U}(\mathcal{O}_X(U),\mathcal{L}(U)),\Delta_U,\epsilon_U)$ of $\mathcal{O}_X/{\mathbb{K}_X}$-bialgebra.
	 From the presheaf $\mathcal{U}(\mathcal{O}_X,\mathcal{L})$,
	 we get a presheaf of associative unital $\mathbb{K}$-algebras and left $\mathcal{O}_X$-module structure on $\mathcal{U}(\mathcal{O}_X,\mathcal{L})\bar{\otimes}_{\mathcal{O}_X}\mathcal{U}(\mathcal{O}_X,\mathcal{L}):=$ $^p(\mathscr{U}(\mathcal{O}_X,\mathcal{L})\bar{\otimes}_{\mathcal{O}_X}\mathscr{U}(\mathcal{O}_X,\mathcal{L}))$. 
Thus, we have a canonical presheaf homomorphisms $$^p \Delta : \mathcal{U}(\mathcal{O}_X,\mathcal{L}) \rightarrow   \mathcal{U}(\mathcal{O}_X,\mathcal{L})\bar{\otimes}_{\mathcal{O}_X}\mathcal{U}(\mathcal{O}_X,\mathcal{L}) $$ with $^p \Delta(U) = \Delta_U$ and	 $^p \epsilon : \mathcal{U}(\mathcal{O}_X,\mathcal{L}) \rightarrow \mathcal{O}_X$ with $^p \epsilon(U) = \epsilon_U$.

Hence, 	$(\mathcal{U}(\mathcal{O}_X,\mathcal{L}), ^p \Delta, ^p \epsilon)$ is a presheaf of $\mathcal{O}_X/\mathbb{K}_X$-bialgebras.
	
It follows that $\mathcal{U}(\mathcal{O}_X,\mathcal{L})\bar{\otimes}_{\mathcal{O}_X}\mathcal{U}(\mathcal{O}_X,\mathcal{L})\hookrightarrow \mathscr{U}(\mathcal{O}_X,\mathcal{L})\bar{\otimes}_{\mathcal{O}_X}\mathscr{U}(\mathcal{O}_X,\mathcal{L})$. By the universal property of the sheafification we get 
	$$\Delta = (^p\Delta)^\# : \mathscr{U}(\mathcal{O}_X,\mathcal{L}) \rightarrow   \mathscr{U}(\mathcal{O}_X,\mathcal{L})\bar{\otimes}_{\mathcal{O}_X}\mathscr{U}(\mathcal{O}_X,\mathcal{L})$$ a morphism between sheaves of associative unital $\mathbb{K}$-algebras and left $\mathcal{O}_X$-modules, 
	$$\epsilon = (^p\epsilon)^\# : \mathscr{U}(\mathcal{O}_X,\mathcal{L}) \rightarrow \mathcal{O}_X $$
	 morphism between sheaves of $\mathcal{O}_X$-modules where 
	 $\Delta \circ i = (\cdot)^\# \circ$ $ ^p \Delta$ and $\epsilon \circ i =$ $^p \epsilon$.
Consequently, $(\mathscr{U}(\mathcal{O}_X,\mathcal{L}),\Delta,\epsilon)$ is a sheaf of $\mathcal{O}_X/{\mathbb{K}_X}$-bialgebras.
	
Observe that $\mathscr{U}(\mathcal{O}_X,\mathcal{L})$ is a cocommutative $\mathcal{O}_X$-coalgebra in the sense that there is a natural cocommutative coassociative comultiplication  \\
$\Delta : \mathscr{U}(\mathcal{O}_X,\mathcal{L}) \rightarrow \mathscr{U}(\mathcal{O}_X,\mathcal{L})\otimes_{\mathcal{O}_X} \mathscr{U}(\mathcal{O}_X,\mathcal{L})$, and counit $\epsilon : \mathscr{U}(\mathcal{O}_X,\mathcal{L}) \rightarrow \mathcal{O}_X$ respectively,
 which are \textbf{locally} given by the following formulas (using the Sweedler convention) \cite{CRV} 
\begin{center}
	$\Delta(f) = f \otimes 1= 1 \otimes f$,\\
	$\Delta(\bar{D}) = \bar{D} \otimes 1 + 1 \otimes \bar{D}$,\\
	$\Delta (D' D'') = \sum D'_{(1)} D''_{(1)} \otimes D'_{(2)} D''_{(2)}$,\\  
	$\epsilon(D) = D(1)$, 
\end{center}
for a section $f$ of $\mathcal{O}_X$, for a section $D$ of $\mathcal{L}$ and for sections $D', D''$ of $\mathscr{U}(\mathcal{O}_X,\mathcal{L})$.
\end{Exm}

\textbf{Remark:} In \cite{MK}, the notion of free Lie algebroid $\mathcal{P}_X$ is constructed for tangent sheaf $\mathcal{T}_X$.  It forms a locally free sheaf of Lie-Rinehart algebras over $(X,\mathcal{O}_X)$ of infinite rank. The universal enveloping algebroid $\mathscr{U}(\mathcal{O}_X,\mathcal{P}_X) =:\mathbb{D}_X$ is described as sheaf of non commutative differential operators on $X$. It is presented therein as a $\mathcal{O}_X$-bialgebra, can be viewed as $\mathcal{O}_X/{\mathbb{K}_X}$-bialgebra defined above.

\subsection{The dual jet space of a Lie algebroid :}  The notion of jet algebroid is  defined as the dual of the universal enveloping algebroid:
\begin{center}
	$\mathscr{J}(\mathcal{O}_X,\mathcal{L}) := \mathscr{H}om_{\mathcal{O}_X} (\mathscr{U}(\mathcal{O}_X,\mathcal{L}), \mathcal{O}_X) $.
\end{center}
 It is sheafification of the presheaf associated to jet algebras $\mathcal{J}(\mathcal{O}_X(U),\mathcal{L}(U))$ of the $(\mathbb{K},\mathcal{O}_X(U))$-Lie-Rinehart algebra $\mathcal{L}(U)$ (see \cite{AP,CV,KP}).
If $\mathcal{L}$ is locally free $\mathcal{O}_X$-module of constant rank, we can view $\mathscr{J}(\mathcal{O}_X,\mathcal{L})$ as projective limit of $n$-jets of $\mathcal{L}$, that is 
\begin{center}
	$\mathscr{J}(\mathcal{O}_X,\mathcal{L}) = \underset{n}{\varprojlim }\mathscr{J}^n(\mathcal{O}_X,\mathcal{L})$  (as $\mathscr{U}(\mathcal{O}_X,\mathcal{L}) = \underset{n}{\varinjlim}$  $\mathscr{U}_{(n)}(\mathcal{O}_X,\mathcal{L}) )$,
\end{center}
where space of $n$-jets of $\mathcal{L}$ is defined as $\mathscr{J}^n(\mathcal{O}_X,\mathcal{L}) := \mathscr{H}om_{\mathcal{O}_X} (\mathscr{U}_{(n)}(\mathcal{O}_X,\mathcal{L}), \mathcal{O}_X)$.\\ 
By definition, $\mathscr{J}(\mathcal{O}_X,\mathcal{L})$ is complete and $\mathscr{U}(\mathcal{O}_X,\mathcal{L})$ is cocomplete with respect to the canonical PBW filtration.

One can dualize the structures on $\mathscr{U}(\mathcal{O}_X,\mathcal{L})$. The product on $\mathscr{J}(\mathcal{O}_X,\mathcal{L})$ is induced from the coproduct of $\mathscr{U}(\mathcal{O}_X,\mathcal{L})$ on each space of sections, which is \textbf{locally} defined by 
	$$(\phi_1 \phi_2) (D) := \phi_1 (D_{(1)}) \phi_2 (D_{(2)})$$ for sections $\phi_1, \phi_2 \in \mathscr{J}(\mathcal{O}_X,\mathcal{L})$ and a section $D \in \mathscr{U}(\mathcal{O}_X,\mathcal{L})$.
By cocommutativity of $\mathscr{U}(\mathcal{O}_X,\mathcal{L})$, this defines a commutative algebra structure on $\mathscr{J}(\mathcal{O}_X,\mathcal{L})$. The unit for this multiplication is \textbf{locally} given by the left counit $\epsilon :\mathscr{U}(\mathcal{O}_X,\mathcal{L}) \rightarrow \mathcal{O}_X$, since
	$$(\epsilon \phi) (D) = \epsilon (D_{(1)}) \phi (D_{(2)}) = \phi(\epsilon (D_{(1)}) D_{(2)})= \phi(D)$$
for a section $\phi$ in $\mathscr{J}(\mathcal{O}_X,\mathcal{L})$ and a section $D$ in $\mathscr{U}(\mathcal{O}_X,\mathcal{L})$.
The left and right $\mathcal{O}_X$-module structure on $\mathscr{U}(\mathcal{O}_X,\mathcal{L})$ provides $\mathcal{O}_X- \mathcal{O}_X$-bimodule structure on $\mathscr{J}(\mathcal{O}_X,\mathcal{L})$. 

Observe now that the product on $\mathscr{U}(\mathcal{O}_X,\mathcal{L})$ descends to a sheaf homomorphism $m : \mathscr{U}(\mathcal{O}_X,\mathcal{L}) \otimes_{\mathcal{O}_X} \mathscr{U}(\mathcal{O}_X,\mathcal{L}) \rightarrow \mathscr{U}(\mathcal{O}_X,\mathcal{L})$. We can therefore dualize the product to obtain a coproduct $\Delta^* : \mathscr{J}(\mathcal{O}_X,\mathcal{L}) \rightarrow \mathscr{J}(\mathcal{O}_X,\mathcal{L}) \otimes_{\mathcal{O}_X} \mathscr{J}(\mathcal{O}_X,\mathcal{L})$ locally defined as
	$$\phi(D D') =: \Delta^*(\phi)(D \otimes D') = \phi_{(1)}(D \phi_{(2)}(D')) $$

Associativity of the multiplication implies that $\Delta^*$ is coassociative. The counit for this coproduct is given by $\epsilon^* : \phi \mapsto \phi(1_{\mathscr{U}(\mathcal{O}_X,\mathcal{L})})$. It follows that $(\mathscr{J}(\mathcal{O}_X,\mathcal{L}), \Delta^*, \epsilon^*)$  is  $\mathcal{O}_X/{\mathbb{K}_X}$-bialgebra.

\vspace{0.5cm}
\textbf{PBW Theorem for jet algebroid :} For a locally free Lie algebroid $\mathcal{L}$ over $(X, \mathcal{O}_X)$ of finite rank, say $r$, locally on some open neighborhood $U_x$ of $x$ we get $$\mathscr{J}(\mathcal{O}_X,\mathcal{L})(U_x) \cong \mathcal{O}_X(U_x) [[w_1, \dots, w_r]]$$
(since $\mathscr{J}(\mathcal{O}_X,\mathcal{L})$ is commutative $\mathcal{O}_X$-algebra and $\mathcal{L^*}$ is locally free $\mathcal{O}_X$-module of rank $r$ and whose local basis on $U_x$ is $\{w_1, \dots, w_r\}$). Thus, $\mathscr{J}(\mathcal{O}_X,\mathcal{L}) \cong \widehat{\mathcal{S}_{\mathcal{O}_X}\mathcal{L^*}}$ (this sheaf of symmetric algebras is formally completed with respect to the degree) as $\mathcal{O}_X$-algebras \cite{CRV,AP}.
\begin{Exm}	
 \textbf{Enveloping algebra as a generalized bialgebra:} 
	Let $R$ be an associative unital $\mathbb{K}$-algebra. Then the enveloping algebra of $R$ is denoted by 
	 $$R^e:=R \otimes_{\mathbb{K}} R^{op}$$  has a canonical left bialgebroid (or $R/{\mathbb{K}}$-bialgebra) structure (see\cite{KP}).

	As a sheaf theoretical generalization we get a $\mathcal{O}_X/{\mathbb{K}_X}$-bialgebra. Here, we consider $\mathcal{O}^e_X$ as the sheafification of the presheaf given by
	$$ U \mapsto \mathcal{O}_X(U)\otimes \mathcal{O}_X(U)^{op} $$ where the restriction map is defined by the component-wise restrictions.
		 
	 The multiplication $m : \mathcal{O}_X^e \otimes_{\mathbb{K}_X} \mathcal{O}_X^e \rightarrow \mathcal{O}_X^e$ locally expressed as
	 $$m((a_1 \otimes b_1),(a_2 \otimes b_2))=a_1 a_2 \otimes b_2 b_1$$ for sections $a_1,a_2$  in $\mathcal{O}_X$ and $b_1,b_2$ in $\mathcal{O}_X^e$. The unit $\eta=i \circ \eta_1:\mathbb{K}_X \rightarrow \mathcal{O}^e_X$ is given by the compositions of the canonical maps $\eta_1:\mathbb{K}_X \rightarrow \mathcal{O}_X$ and $i:\mathcal{O}_X \hookrightarrow \mathcal{O}^e_X$.
	
	 The comultiplication and counit are locally defined as follows: 
$$\Delta(a \otimes b)=(a \otimes 1)\otimes_{\mathcal{O}_X} (1 \otimes b)$$ 
	  and $\epsilon(a \otimes b)=ab $  for sections $a,b$ in $\mathcal{O}_X^e$.
	  
More generally, if we consider any $\mathbb{K}_X$-algebra $\mathcal{A}$, the sheafification of the presheaf of enveloping algebras $\mathcal{A}^e:= \mathcal{A}\otimes \mathcal{A}^{op}$ forms a natural $\mathcal{A}/{\mathbb{K}_X}$-bialgebra structures whose local descriptions is as given above. 
\end{Exm}

\subsection{Primitive filtration of a $\mathcal{O}_X/{\mathbb{K}_X}$-bialgebra :}\label{primitive}
For a given $\mathcal{O}_X/{\mathbb{K}_X}$-bialgebra $\mathcal{A}$, it follows that $\bar{\mathcal{A}} :=$ $\mathscr{K}er$ $\epsilon$ is a subsheaf of $  \mathcal{O}_X$-modules and $\mathcal{A} = \mathcal{O}_X \oplus \bar{\mathcal{A}}$ as direct sum of $\mathcal{O}_X$-modules. Moreover, the $\mathcal{O}_X$-submodule, $\bar{\mathcal{A}}$ is a subsheaf of $\mathbb{K}_X$-algebras equipped with cocommutative coassociative coproduct
$$\bar{\Delta} : \bar{\mathcal{A}} \rightarrow \bar{\mathcal{A}} \otimes_{\mathcal{O}_X} \bar{\mathcal{A}}$$ is locally given by $\bar{\Delta}(s) =\Delta(s) -  (s \otimes 1 + 1 \otimes s)$ for  a section $s$ in $\bar{\mathcal{A}}$. Consequently, $\bar{\mathcal{A}}$ is a non-counital $\mathcal{O}_X$-coalgebra.

Here, the bialgebra  $\mathcal{A}$ can be reconstructed from $\mathcal{O}_X$-coalgebra $\bar{\mathcal{A}}$ with the coproduct $ \bar{\Delta}$ and the structure of sheaf of $\mathbb{K}_X$-algebras on $\bar{\mathcal{A}}$. Suppose $\bar{\Delta}^{(n)} $  the iterated coproduct between $\mathcal{O}_X$-modules where $$\bar{\Delta}^{(1)} = \bar{\Delta} ~~\mbox{and}~~ \bar{\Delta}^{(n)} = (\bar{\Delta} \otimes \mbox{Id}_{\bar{\mathcal{A}}}) \circ \bar{\Delta}^{(n-1)}~~\mbox{for}~~ n\geq 2.$$
If 
$\bar{\mathcal{A}}_n =$ $\mathscr{K}$er $\bar{\Delta}^{(n)}$ then we get
a filtration (of $\mathcal{O}_X$-modules $\bar{\mathcal{A}}$) written as follows: 
$${0} = \bar{\mathcal{A}}_0 \subset \bar{\mathcal{A}}_1 \subset \bar{\mathcal{A}}_2 \subset \cdots .$$
It is refered as the primitive filtration of $\bar{\mathcal{A}}$. We write ${\mathcal{A}}_n = \mathcal{O}_X \oplus \bar{\mathcal{A}}_n $ for $ n \geq 0$.
 The subsheaf of $\mathcal{O}_X$-modules $\bar{\mathcal{A}}_1$ is called (sub)sheaf of primitive elements and is also denoted by $\mathscr{P}(\mathcal{A})$.
The bialgebra $\mathcal{A}$ is said to be cocomplete 
if $\mathcal{A} = \cup_{n \geq 0} {\mathcal{A}}_n$.
It is called locally graded  free if each of the sheaf of subquotients 
$${\mathcal{A}}_{n+1}/{{\mathcal{A}}_n} = {\bar{\mathcal{A}}}_{n+1}/{{\bar{\mathcal{A}}}_n}$$ is a locally free $\mathcal{O}_X$-modules. For this we have  $\mathscr{P}(\mathcal{A})$ is a locally free $\mathcal{O}_X$-module
We say that $\mathcal{A}$ is of finite type if $\mathscr{P}(\mathcal{A})$ is locally free $\mathcal{O}_X$-module of finite rank.

\begin{Rem}
The universal enveloping algebra of a Lie algebroid over a smooth manifold is cocomplete, graded projective and finite type (see \cite{MM}). 
Here, we consider the presheaf $\mathcal{U}(\mathcal{O}_X, \mathcal{L})$ for a Lie algebroid $\mathcal{L}$ with $\mathcal{L}$ is a locally free left $\mathcal{O}_X$-module. It follows that 
the presheaf of natural filtration and primitive filtration coincides on each special open set $U_x$ for the presheaf $\mathcal{U}(\mathcal{O}_X, \mathcal{L})$. There is an isomorphism of $\mathcal{O}_{X,x}$-modules between the induced stalks of the two types of presheaves. On application of sheafification functor we obtain sheaf isomorphism between the filtrations of the sheaf $\mathscr{U}(\mathcal{O}_X, \mathcal{L})$.
The universal enveloping algebroid $\mathscr{U}(\mathcal{O}_X, \mathcal{L})$ is cocomplete, locally graded free. Furthermore,  if $\mathcal{L}$ is of finite rank then $\mathscr{U}(\mathcal{O}_X, \mathcal{L})$ is of finite type.
\end{Rem}
\subsection{Sheaf of Lie-Rinehart algebras of primitive elements :}
Let $\mathcal{A}$ be a $\mathcal{O}_X/\mathbb{K}_X$-bialgebra. We want to construct its associated Lie algebroid $(\mathscr{P}(\mathcal{A}), [\cdot,\cdot]_c, \mathfrak{a})$ over an $a$- space $(X, \mathcal{O}_X)$.

The counit map $\epsilon : \mathcal{A}\rightarrow \mathcal{O}_X$ induces  sheaf homomorphism $\rho : \mathcal{A} \rightarrow \mathscr{E}nd_{\mathbb{K}_X}(\mathcal{O}_X)$ of $\mathbb{K}_X$-algebras. This map is locally expressed as 
$\rho(s)(f) = \epsilon(s.f)$
for a section $s \in \mathcal{A}$ and a section $f \in \mathcal{O}_X$. Thus, for each open set $U$ of $X$ we get the $(\mathbb{K},\mathcal{O}_X(U))$-Lie-Rinehart algebra $(\mathcal{P}(\mathcal{A}(U)),[\cdot,\cdot]_c,\mathfrak{a}_U)$ consisting with primitive elements of $\mathcal{O}_X/{\mathbb{K}}$-bilagebra $\mathcal{A}(U)$. The anchor map $$\mathfrak{a}_U : (\mathcal{P}(\mathcal{A}(U)),[\cdot,\cdot]_c)\rightarrow (Der_{\mathbb{K}}(\mathcal{O}_X(U)),[\cdot,\cdot]_c)$$ is defined by 
$\mathfrak{a}_U(D)(f) = \epsilon (U)(D \cdot f)$ for $f\in \mathcal{O}_X(U)$ and $D\in \mathcal{P}({\mathcal{A}(U)})$.

Now, for a pair of open sets $V,U$ in $X$ with $V \subset U$  we have 
the restriction morphism  $res_{UV}^{\mathcal{A}}  : \mathcal{A}(U) \rightarrow \mathcal{A}(V) $. It  
induces a morphism $res_{UV}^{\mathscr{P}(\mathcal{A})}$ between the $(\mathbb{K},\mathcal{O}_X(U))$ Lie-Rinehart algebras $(\mathcal{P}(\mathcal{A}(U)), [\cdot,\cdot]_c,\mathfrak{a}_U)$ and $(\mathbb{K},\mathcal{O}_X(V))$ Lie-Rinehart algebra $(\mathcal{P}(\mathcal{A}(V)),[\cdot,\cdot]_c,\mathfrak{a}_V) $.
Then we get $res_{UV}^{\mathcal{T}} \circ \mathfrak{a}_U = \mathfrak{a}_V \circ res_{UV}^{\mathscr{P}(\mathcal{A})}$.
Thus, it follows that
\begin{center}
	$\mathscr{P}(\mathcal{A}): {Open}_X \rightarrow \mathbb{K}$-Lie Algebras\\
	$U \mapsto (\mathcal{P}(\mathcal{A}(U)),[\cdot,\cdot]_c,\mathfrak{a}_U)$
\end{center}
forms a sheaf of Lie-Rinehart algebras with the restriction morphism $ res_{UV}^{\mathscr{P}(\mathcal{A})}$ (the sheaf structure on $\mathscr{P}(\mathcal{A})$ induces canonically from the sheaf $\mathcal{A}$).

Thus $(\mathscr{P}(\mathcal{A}), [\cdot,\cdot]_c, \mathfrak{a})$ is a Lie algebroid over $(X, \mathcal{O}_X)$ 
with anchor map 
$$\mathfrak{a} : (\mathscr{P}(\mathcal{A}), [\cdot,\cdot])
\rightarrow (\mathcal{D}er_{\mathbb{K}_X}(\mathcal{O}_X), [\cdot,\cdot]_c)$$ is defined by
$\mathfrak{a}(U) = \mathfrak{a}_U$ on the space of sections.


\vspace{0.5cm}
\textbf {Primitive elements for Weyl algebra associated with normal crossing:}
Consider the $\mathbb{K}$-algebra $\tilde{R}:= \mathbb{K}[x,y]$. The $(\mathbb{K},\tilde{R})$-Lie-Rinehart algebra of derivations  $Der_{\mathbb{K}}(\tilde{R})=\langle \{\partial_x, \partial_y\} \rangle$.
Thus the Weyl algebra of $\tilde{R}$ is 
$$\mathcal{D}iff(\tilde{R})\cong \mathcal{U}(\tilde{R}, Der_{\mathbb{K}}(\tilde{R}))$$
forms a $\tilde{R}/{\mathbb{K}}$-bialgebra. Consider the $\mathbb{K}$-algebra $R:=\mathbb{K}[x,y]/\langle xy \rangle$. The associated space of differential operators $\mathcal{D}iff(R)$ and the associated universal enveloping algebra $\mathcal{U}(R, Der_{\mathbb{K}}(R))$ with $Der_{\mathbb{K}}(R)=\langle \{x\partial_x, y\partial_y \} \rangle$, forms a $R/{\mathbb{K}}$-bialgebras.

For the first case $Der_{\mathbb{K}}(\tilde{R})$ is a projective (in fact free) $\tilde{R}$-module of rank $2$, by using Cartier-Milnor-Moore theorem for Lie-Rinehart algebra \cite{MM}, we get the space of primitive elements of $\mathcal{D}iff(\tilde{R})$ is $\mathcal{P}(\mathcal{U}(\tilde{R}, Der_{\mathbb{K}}(\tilde{R})))\cong Der_{\mathbb{K}}(\tilde{R})$. But, this approach will not work for the second case, since there $Der_{\mathbb{K}}(R)$ is not a projective module over the base algebra $R$.

Here we compute the space of primitive elements $\mathcal{P}(\mathcal{D}iff(R))$. 
First note that $Der_{\mathbb{K}}(R)\subset \mathcal{P}(\mathcal{D}iff(R))$. Here $\Delta_R(x \partial_x)=x \partial_x\otimes_R 1+ 1 \otimes_R x\partial_x$ and similarly $\Delta(y \partial_y)=y \partial_y\otimes_R 1+ 1 \otimes_R y\partial_y$.
Consider the second order differential operator $x \partial_x \circ y \partial_y$, then
$\Delta(x \partial_x \circ y \partial_y)=x\partial_x \circ y \partial_y \otimes_R 1+1 \otimes_R x \partial_x \circ y \partial_y$ (as $y \partial_y \otimes_R x \partial_x = 0 = x \partial_x \otimes_R y \partial_y)$ and thus $\mathcal{D}iff_{(2)}(R) \setminus R \subset \mathcal{P}(\mathcal{D}iff(R))$ where $\mathcal{D}iff_{(2)}(R)$
is the second filtered $R$-module for the natural filtration of $\mathcal{D}iff(R)$ (like for the case $\mathcal{U}(R, Der_{\mathbb{K}}(R))$). Note that $y \partial_y \circ x \partial_x = x \partial_x \circ y \partial_y -[x \partial_x, y \partial_y] \in \mathcal{P}(\mathcal{U}(R, Der_{\mathbb{K}}(R))$.
Consequently, $Der_{\mathbb{K}}(R)\subsetneq \mathcal{P}(\mathcal{D}iff(R))$ and 
thus, $Der_{\mathbb{K}}(R)$ is non-isomorphic to $\mathcal{P}(\mathcal{U}(R, Der_{\mathbb{K}}(R))$.

This idea can be extended to its associated complex geometric situation.

\section{Generalized Cartier-Milnor-Moore Theorem} 

Let us fix an special $a$-space $(X, \mathcal{O}_X)$ and consider the following three categories. 
\begin{itemize}
\item   $\mathcal{C}$ be the category of Lie algebroids over $(X,\mathcal{O}_X)$,  
\item $\bar{\mathcal{D}}$ denotes the category of presheaves of $\mathcal{O}_X/{\mathbb{K}_X} $-bialgebras, and
\item  $\mathcal{D}$ denotes the category of $\mathcal{O}_X/{\mathbb{K}_X} $-bialgebras.
\end{itemize}
 \subsection{$\mathscr{U}(\mathcal{O}_X,-)$ or simply $\mathscr{U}$ as a functor :} Following notations in previous sections, the assignment $\mathcal{L} \mapsto \mathcal{U}(\mathcal{O}_X, \mathcal{L})$ defines a covariant functor $$\mathcal{U}:\mathcal{C} \rightarrow \bar{\mathcal{D}}.$$ Then we define the functor $\mathscr{U}: \mathcal{C}\rightarrow \mathcal{D}$ by composition with the sheafification functor $$(\cdot)^\#: \bar{\mathcal{D}} \rightarrow \mathcal{D}.$$

Here we proceed as follows for morphisms, in order to define the composition of the functor $\mathcal{U}$ followed by sheafification.
Let $\phi : (\mathcal{L}_1, [\cdot,\cdot]_1, \mathfrak{a}_1) \rightarrow (\mathcal{L}_2, [\cdot,\cdot]_2, \mathfrak{a}_2) $ be a morphism in $\mathcal{C}$. We find a morphism $\widetilde {^p \phi}$ in the category $\bar{\mathcal{D}}$ where $$\widetilde {^p \phi} : \mathcal{U}(\mathcal{O}_X, \mathcal{L}_1) \rightarrow \mathcal{U}(\mathcal{O}_X, \mathcal{L}_2)$$  is defined by $\widetilde {^p \phi}(U) = \widetilde {\phi(U)}$ for every open set $U$ of $X$. Note that $\widetilde{{\phi(U)}}$ is obtained by evaluating the functor $\mathcal{U}(\mathcal{O}_X(U), -)$ on the morphism $\phi(U)$.
 
Finally we obtain the morphism of sheaves of $\mathcal{O}_X/{\mathbb{K}_X} $-bialgebras denoted by ${\widetilde {\phi}}:\mathscr{U}(\mathcal{O}_X, \mathcal{L}_1) \rightarrow \mathscr{U}(\mathcal{O}_X, \mathcal{L}_2) $ where $\widetilde{\phi}= (\cdot)^\# \circ \widetilde {^p \phi}$. 

\subsection{$\mathscr{P}(-)$ or simply $\mathscr{P}$ as a functor :}  Following notations in previous sections, we show that the assignment $\mathcal{A} \mapsto \mathscr{P}(\mathcal{A})$ defines a covariant functor 
$$\mathscr{P} : \mathcal{D} \rightarrow  \mathcal{C}.$$
We have seen before that if $\mathcal{A}$ is an object in  the category $\mathcal{D}$ then $\mathscr{P}(\mathcal{A})$ is an object in the category $\mathcal{C}$.
Let $\psi : \mathcal{A} \rightarrow \mathcal{\tilde{A}}$ be a morphism in the category $\mathcal{D}$.
Then, on each open set $U$ of $X$, we have the morphism of $(\mathbb{K}, \mathcal{O}_X(U))$-Lie-Rinehart algebras
$$\phi_U := \psi(U)|_{\mathcal{P}(\mathcal{A}(U))} : \mathcal{P}(\mathcal{A}(U))\rightarrow \mathcal{P}(\mathcal{\tilde{A}}(U)),$$ which are compatible with the restriction morphisms.

Thus we obtain a morphism of Lie algebroids $\phi:\mathscr{P}(\mathcal{A}) \rightarrow \mathscr{P}(\mathcal{\tilde{A}})$ defined by $\phi(U)=\phi_U$.
 

Next, we use results from \cite{MM} as the local descriptions.
\begin{Note} \label{pri of L}
		The canonical maps (\ref{Uni env alg}) 
		\begin{equation*}	
		\iota_{\mathcal{L}}:\mathcal{L} \rightarrow  \mathscr{U}(\mathcal{O}_X, \mathcal{L}) ~~\mbox{and}~~
		\mathscr{U}(\mathcal{O}_X, \mathcal{L}) \rightarrow \mathscr{P}(\mathscr{U}(\mathcal{O}_X, \mathcal{L}))
		\end{equation*} obtained from the functor $\mathscr{P}$, provides a morphism $\alpha_{\mathcal{L}}:\mathcal{L} \rightarrow \mathscr{P}(\mathscr{U}(\mathcal{O}_X, \mathcal{L}))$ in the category $\mathcal{C}$ (see Example \ref{universal alg}) via compositions of the maps. On each open set $U$, we have the canonical homomorphism $\alpha_{\mathcal{L}(U)}:\mathcal{L}(U)\rightarrow \mathcal{P}(\mathcal{U}(\mathcal{O}_X(U),\mathcal{L}(U)))$ compatible with restriction morphism, and sheafification provides the map $\alpha_{\mathcal{L}}$.
	
	In addition, if $\mathcal{L}$ is locally free $\mathcal{O}_X$-module, the morphism $\alpha_{\mathcal{L}}$ become an isomorphism i.e., $\mathcal{L} \cong \mathscr{P}(\mathscr{U}(\mathcal{O}_X, \mathcal{L}))$.
\end{Note} \label{uni of A}
\begin{Note}
	For an object $\mathcal{A}$ in the category $\mathcal{D}$, the inclusion map $\mathscr{P}(\mathcal{A}) \hookrightarrow \mathcal{A}$ provides a morphism $$\beta_{\mathcal{A}}: \mathscr{U}(\mathcal{O}_X,\mathscr{P}(\mathcal{A})) \rightarrow \mathcal{A}$$ in $\mathcal{D}$ by using the universal property of the universal enveloping algebroid (Remark \ref{Uni prop}) $\mathscr{U}(\mathcal{O}_X,\mathscr{P}(\mathcal{A}))$. On each open set $U$, we have the canonical homomorphism $\beta_{\mathcal{A}(U)} : \mathcal{U}(\mathcal{O}_X(U), \mathcal{P}(\mathcal{A}(U))) \rightarrow \mathcal{A}(U)$ compatible with restriction morphism, provides the map $\beta_\mathcal{A}$
	
	In addition, if $\mathcal{A}$ is cocomplete locally graded free $\mathcal{O}_X$-module, the morphism $\beta_{\mathcal{A}}$ become an isomorphism i.e., $ \mathcal{A} \cong \mathscr{U}(\mathcal{O}_X,\mathscr{P}(\mathcal{A}))$. 
\end{Note}
\subsection{Cartier-Milnor-Moore Theorem for Lie algebroids over $a$-spaces}
Here we consider a generalized version of Cartier-Milnor-Moore Theorem (CMM theorem in short hand notation) in the sheaf theoretic terminology.

 We denote by $C$ the category of $(\mathbb{K},R)$-Lie-Rinehart algebras and the category of $R/{\mathbb{K}}$-bialgebras by $D$. In terms of the notations previously used, we consider the next result. 
\begin{Thm}
\begin{enumerate}
\item The functor $\mathscr{U}$ is left adjoint to the functor $\mathscr{P}$.
\item  

The functors $\mathscr{U}$ and $\mathscr{P}$ restrict to an equivalence between the full subcategory $\mathcal{C'}$ of locally free Lie algebroids over $(X,\mathcal{O}_X)$  
	and the full subcategory $\mathcal{D'}$ of cocomplete locally graded  free $\mathcal{O}_X/\mathbb{K}_X$-bialgebras. 
\end{enumerate}	
\end{Thm}
\begin{proof}(1) Let $\phi : \mathcal{L} \rightarrow \mathscr{P}(\mathcal{A})$ be a morphism in the category $\mathcal{C}$ of Lie algebroids over $(X,\mathcal{O}_X)$.
For a pair of open sets $V, U$ in $X$ where $V \subset U$ we have the morphisms  $\phi(U)$ and $\phi(V)$ in the category $C$ with  $ res_{UV}^{\mathscr{P}(\mathcal{A})} \circ \phi(U) = \phi(V) \circ res_{UV}^{\mathcal{L}}$.

As a byproduct of the universal enveloping algebra functor we get morphisms 
		$$\widetilde{\phi(U)} : \mathcal{U}(\mathcal{O}_X(U), \mathcal{L}(U)) \rightarrow \mathcal{A}(U)~~\mbox{and}~~
		\widetilde{\phi(V)} : \mathcal{U}(\mathcal{O}_X(V), \mathcal{L}(V)) \rightarrow \mathcal{A}(V)$$
	with the compatibility condition $ res_{UV}^{\mathcal{A}} \circ \widetilde{\phi(U)} = \widetilde{\phi(V)} \circ res_{UV}^{\mathcal{U}}$.
		Then $$\widetilde{\phi} : \mathcal{U}(\mathcal{O}_X, \mathcal{L}) \rightarrow \mathcal{A}$$ is a morphism between presheaf of $\mathcal{O}_X/\mathbb{K}_X$-bialgebras,  where $\widetilde{\phi}(U) : =\widetilde{\phi(U)}$ for every open set $U$ of $X$.
	By the universal property of sheafification we find a morphism 
	$\bar{\phi} : \mathscr{U}(\mathcal{O}_X ,\mathcal{L})\rightarrow \mathcal{A}$
	in $\mathcal{D}$ where $\bar{\phi} $ is locally given by the map $\widetilde{\phi}$.
	
On the other hand, suppose  $\psi : \mathscr{U}(\mathcal{O}_X,\mathcal{L}) \rightarrow \mathcal{A}$ is a morphism in $\mathcal{D}$.
	Then for any pair of open sets $V, U$ of $X$ where $V \subset U$ we have the morphisms 
	 $\psi(U)$ and $\psi(V)$ in the category $D$ with $ res_{UV}^{\mathcal{A}} \circ \psi(U) = \psi(V) \circ res_{UV}^{\mathscr{U}}$.

	The primitive elements functor gives (using Note \ref{pri of L}) the morphisms in $C$ written as $\phi_U : \mathcal{L}(U) \rightarrow \mathcal{P}(\mathcal{A}(U))$ and  $\phi_V : \mathcal{L}(V) \rightarrow \mathcal{P}(\mathcal{A}(V))$ respectively. 
	
It follows by using the functor $\mathscr{P}$ that $ res_{UV}^{\mathscr{P}(\mathcal{A})} \circ \phi_U = \phi_V \circ res_{UV}^{\mathcal{L}}$.
	
	Thus, we can form $\phi : \mathcal{L} \rightarrow \mathscr{P}(\mathcal{A})$, a morphism in $\mathcal{C}$ where $\phi(U) := \phi_U.$ 	
Consequently, we get a bijection of the sets $\mathscr{H}om_{\mathcal{C}}(\mathcal{L},\mathscr{P}(\mathcal{A}))$ and $\mathscr{H}om_{\mathcal{D}}(\mathscr{U}(\mathcal{O}_X ,\mathcal{L}), \mathcal{A})$.
 	
	For an object $\mathcal{L}$ in the category $\mathcal{C}$, we have the canonical morphism $$\alpha_{\mathcal{L}} : \mathcal{L} \rightarrow \mathscr{P}(\mathscr{U}(\mathcal{O}_X, \mathcal{L}))$$ in $\mathcal{C}$ (Note \ref{pri of L}). 

Let $\phi : \mathcal{L}_1 \rightarrow \mathcal{L}_2$ be a morphism in category $\mathcal{C}$. We find the sheaf homomorphisms $\alpha_{\mathcal{L}_1}$,  $\alpha_{\mathcal{L}_2}$ and $\widetilde{\phi}_{res}:=\mathscr{P}(\tilde{\phi})$, which are morphisms in $\mathcal{C}$ such that $\tilde{\phi}_{res} \circ \alpha_{\mathcal{L}_1} = \alpha_{\mathcal{L}_2} \circ \phi$.

 It defines unit of the adjunction defined by a natural transformation 
	\begin{center}
		$\alpha : Id_{\mathcal{C}} \rightarrow \mathscr{P} \circ \mathscr{U}$ where
		$\alpha_{\mathcal{L}}:\mathcal{L} \mapsto \mathscr{P}(\mathscr{U}(\mathcal{O}_X, \mathcal{L}))$.
	\end{center}
 Analogously for  a morphism $\psi: \mathcal{A}_1 \rightarrow \mathcal{A}_2$ in the category $\mathcal{D}$, we have compatible sheaf homomorphisms $\beta_{\mathcal{A}_1}, \beta_{\mathcal{A}_2}$ (Note \ref{uni of A}) and 
 $$\widetilde{\psi_{res}}:=\mathscr{U}\circ \mathscr{P}~(\psi).$$ 
 As a result, the counit of the adjunction is given by a natural transformation 
	\begin{center}
		$\beta :  \mathscr{U}\circ \mathscr{P}  \rightarrow Id_{\mathcal{D}}$ where 
		$ \beta_{\mathcal{A}}: \mathscr{U}(\mathcal{O}_X, \mathscr{P}(\mathcal{A})) \mapsto \mathcal{A}$.
	\end{center}
	The underlying presheaf homomorphisms satisfy the triangular identities on the space of sections of objects $\mathcal{L}$ in $\mathcal{C}$ and $\mathcal{A}$ in $\mathcal{D}$ respectively. Here we have
	\begin{center}
		$\mathcal{P}(\beta_{\mathcal{A}(U)}) \circ \alpha_{\mathcal{P}(\mathcal{A}(U))} = id_{\mathcal{P}(\mathcal{A}(U))}$\\
		and $\beta_{\mathcal{U}(\mathcal{O}_X(U), \mathcal{L}(U))} \circ \mathcal{U}(\mathcal{O}_X(U), \alpha_{\mathcal{L}(U)}) = id_{\mathcal{U}(\mathcal{O}_X(U), \mathcal{L}(U))}$.
	\end{center}
	By sheafification, it induces morphism in $\mathcal{C}$  and $\mathcal{D}$ respectively satisfying 
	\begin{center}
		$\mathscr{P}(\beta_{\mathcal{A}}) \circ \alpha_{\mathscr{P}(\mathcal{A})} = id_{\mathscr{P}(\mathcal{A})}$ and
		 $\beta_{\mathscr{U}(\mathcal{O}_X, \mathcal{L})} \circ \mathscr{U}(\mathcal{O}_X, \alpha_{\mathcal{L}}) = id_{\mathscr{U}(\mathcal{O}_X, \mathcal{L})}$.
	\end{center}
	(2)	Now, we have a special open cover $\{U_x\}_x$ of $X$ for which $\mathcal{L}|_{U_x}$ is free as left $\mathcal{O}_X|_{U_x}$-module. For any open set $V$ inside a special open set $U_x$, 
  we get the isomorphism 
	$$\alpha_{\mathcal{L}(V)} : \mathcal{L}(V)  \rightarrow \mathcal{P}(\mathcal{U}(\mathcal{O}_X(V), \mathcal{L}(V)))$$ in $C$. This is also compatible with the restriction maps and induces stalk-wise isomorphisms.
	Thus, $\alpha_{\mathcal{L}}$ is an isomorphism (using Note \ref{pri of L}, Lemma \ref{iso on sheaf level}) for $\mathcal{L}$ in the full subcategory of locally free Lie algebroids over $(X,\mathcal{O}_X)$.
	We get the unit of the adjunction $\alpha : Id_{\mathcal{C'}} \rightarrow \mathscr{P} \circ \mathscr{U}$  is a natural isomorphism, or, $Id_{\mathcal{C'}} \cong \mathscr{P} \circ \mathscr{U}$. 
	

In this case too, we have the special open cover $\{U_x\}_x$ of $X$ for which $\mathcal{A}(V)$ is cocomplete graded free $\mathcal{O}_X(V)/\mathbb{K}$-bialgebra for any open set $V$ inside some $U_x$. Thus, we get the isomorphism 
$$\beta_{{\mathcal{A}}(V)} : \mathcal{U}(\mathcal{O}_X(V), \mathcal{P}(\mathcal{A}(V))) \rightarrow \mathcal{A}(V)$$
in $D$, compatible with restriction maps. It induces stalk-wise isomorphism and provides $\beta_{\mathcal{A}}$ is an isomorphism (using Note \ref{uni of A}, Lemma \ref{iso on sheaf level}) for $\mathcal{A}$ in the full subcategory of cocomplete graded free $\mathcal{O}_X/{\mathbb{K}_X}$-bialgebras. Thus, the counit 
	$\beta : \mathscr{U}\circ \mathscr{P} \rightarrow Id_{\mathcal{D'}} $ of the adjunction is a natural isomorphism, or, $\mathscr{U} \circ \mathscr{P} \cong Id_{\mathcal{D'}}$.
	
\end{proof}
\begin{Cor}
	The functors $\mathscr{U}$ and $\mathscr{P}$ restrict to an equivalence between the full subcategory of locally free Lie algebroids over $(X,\mathcal{O}_X)$ of finite ranks
and the full subcategory of cocomplete locally graded free $\mathcal{O}_X/\mathbb{K}_X$-bialgebras of finite type.
	

\end{Cor}

\begin{Rem}

One may work on locally free Lie algebroids $\mathcal{L}$ (of finite rank) over noetherian separated schemes (or schemes of finite types) over a field of characteristic zero \cite{UB,MK}.
The PBW theorem and the CMM theorem also hold in this case. On each affine pieces $U_i$, the $(\mathbb{K}, \mathcal{O}_{X}(U_i))$-Lie-Rinehart algebra $\mathcal{L}(U_i)$ is projective $\mathcal{O}_X(U_i)$-module and the stalks ${\mathcal{L}}_x$ are free $(\mathbb{K}, \mathcal{O}_{X,x})$-Lie-Rinehart algebras for every $x \in X$ (using Lemma 5.1.3. from\cite{CW}). We find the PBW isomorphisms of $\mathcal{O}_X(U_i)$-algebras
	$S_{\mathcal{O}_{X}(U_i)}\mathcal{L}(U_i) \overset{\theta_{U_i}}{\longrightarrow} \mbox{gr}(~\mathcal{U}(\mathcal{O}_{X}(U_i), \mathcal{L}(U_i))~), $
 which are compatible with restrictions. It provides PBW isomorphisms for $\mathcal{L}_x$ on the level of stalks and also the generalized PBW isomorphism.

Moreover, if we consider a locally free Lie algebroid of infinite rank, then the PBW isomorphisms on the level of stalks provides the generalized PBW isomorphism.
\end{Rem}
{\bf Acknowledgements:} \\
We acknowledge the research support from C. S. I. R fellowship grant $2015$ and SERB, DST, Government of India, MATRICS -Research grant  $MTR/2018/000510$. 
We sincerely thank the anonymous referee for making several useful and important suggestions, which helped to improve the  presentation and clarity of the article. 


\vspace{.25cm}
{\bf Ashis Mandal and  Abhishek Sarkar}\\
 Department of Mathematics and Statistics,\\
Indian Institute of Technology,
Kanpur 208016, India.\\
e-mail: amandal@iitk.ac.in, ~~ abhishsr@iitk.ac.in

\begin{thebibliography}{ref}
	
\bibitem{AP}
Blom, A., Posthuma, H. (2015).
\newblock An index theorem for Lie algebroids. 
 arXiv:1512.07863 [math.QA].

\bibitem{UB}
Bruzzo, U. (2017).
\newblock Lie algebroid cohomology as a derived functor.
\newblock{\em J. Algebra}. 483 : 245-–261. DOI: 10.1016/j.jalgebra.2017.03.030

	\bibitem{BRT}
Bruzzo, U., Mencattini, I., Rubtsov, V. N., Tortella, P. (2015).
\newblock Nonabelian holomorphic Lie algebroid extensions.
\newblock {\em  Internat. J. Math.} 26 (5):1550040. DOI: 10.1142/S0129167X15500408

\bibitem{CV}
Calaque, D., Van den Bergh, M. (2010).
\newblock Hochschild cohomology and Atiyah classes. 
\newblock {\em Adv. Math.}, 224 (5): 1839–-1889. DOI: 10.1016/j.aim.2010.01.012

\bibitem{CRV} 
Calaque, D., Rossi, C. A., Van den Bergh, M. (2010).
\newblock Hochschild (Co)Homology for Lie algebroids.
\newblock{\em Int. Math. Res. Not. IMRN.} (21): 4098 -- 4136. DOI: 10.1093/imrn/rnq033
	
\bibitem{VD}
Drinfeld, V. (2006)
\newblock Infinite-dimensional vector bundles in algebraic geometry: an introduction. The unity of mathematics.
\newblock {\em Progr. Math.} 244: 263–-304. DOI: 10.1007/0-8176-4467-9\_7

\bibitem{RF}
Fernandes, R. L. (2002).
\newblock Lie algebroids, Holonomy and Characteristic classes.
\newblock{\em Adv. Math.} 170 (1): 119--179. DOI: 10.1006/aima.2001.2070

\bibitem{FF}
Forstneri$\check{c}$, F. (2011).
\newblock Stein Manifolds and Holomorphic Mappings: The Homotopy Principle in Complex Analysis.
\newblock {\em Springer, Heidelberg}. 56. DOI: 10.1007/978-3-642-22250-4


\bibitem{MK}
Kapranov, M. (2007).
\newblock  Free Lie algebroids and the space of paths. 
\newblock {\em Selecta Math. (N.S.)}. 13 (2): 277--319. DOI: 10.1007/s00029-007-0041-9

\bibitem{RK}
Klaasse, R. L. (2011).
\newblock Geometric structures and Lie algebroids. 
arXiv:1712.09560 [math.SG]. Thesis (Ph.D.)-University of Utrecht

\bibitem{KP}
Kowalzig, N., Posthuma, H. (2011).
\newblock The cyclic theory of Hopf algebroids.
\newblock{\em J. Noncommut. Geom.} 5(3): 423–-476. DOI: 10.4171/JNCG/82

\bibitem{CL}
 Laurent-Gengoux, C., Ryvkin, L. (2021).
\newblock The neighborhood of a singular leaf.
\newblock{\em J. \'{E}c. polytech. Math.} 8: 1037--1064. DOI: 10.5802/jep.165



\bibitem{HM}
Maakestad, H. (2007).
\newblock Chern classes and Lie-Rinehart algebras.
\newblock{\em Indag. Math. (N.S.)} 18(4): 589--599. DOI: 10.1016/S0019-3577(07)80065-6

\bibitem{KM}
Mackenzie, K. C. H. (2005).
\newblock General theory of Lie groupoids and Lie algebroids.
London Mathematical Society Lecture Note Series. 213.
\newblock {\em  Cambridge University Press, Cambridge.} DOI: 10.1017/CBO9781107325883

\bibitem{MM}
Moerdijk, I., Mr\v{c}un, J. (2010).
\newblock On the universal enveloping algebra of a Lie algebroid.
\newblock {\em Proc. Amer. Math. Soc.} 138(9): 3135-–3145. DOI: 10.1090/S0002-9939-10-10347-5

\bibitem{AM}
Mukherjee, A. (2015).
\newblock Differential Topology.
Second edition.
\newblock {\em Hindustan Book Agency, New Delhi; Birkh\"{a}user/Springer, Cham}. DOI: 10.1007/978-3-319-19045-7

\bibitem{LM}
Narv$\acute{a}$ez Macarro, L.  (2014).
\newblock Differential Structures in Commutative Algebra.
\newblock{\em Mini-course at the XXIII Brazilian Algebra Meeting, July 27 - August 1}, Maring$\acute{a}$, Brazil.



\bibitem{BP}
 Pym, B. (2013)
\newblock Poission Structures and Lie Algebroids in Complex Geometry.
\newblock {\em ProQuest LLC, Ann Arbor, MI}. Thesis (Ph.D.)–University of Toronto (Canada)

\bibitem{SR}
Ramanan, S. (2005).
\newblock {Global Calculus.}
Graduate Studies in Mathematics. 65.
\newblock {\em American Mathematical Society, Providence, RI}.  DOI: 10.1090/gsm/065

\bibitem{GR}
Rinehart, G. S. (1963).
\newblock {Differential forms on general commutative algebras.}
\newblock {\em Trans. Amer. Math. Soc.} 108: 195–-222. DOI: 10.2307/1993603

\bibitem{TS}
Schedler, T. (2019).
\newblock Deformation of algebras in noncommutative geometry.
 arXiv:1212.0914v3 [math.RA]



\bibitem{CW}
 Weibel, C. A. (2013)
\newblock The K-book. An introduction to algebraic K-theory. 
Graduate Studies in Mathematics, 145.
\newblock {\em  American Mathematical Society, Providence, RI}. DOI: 10.1090/gsm/145
	

\end{thebibliography}
\end{document}